\documentclass[12pt]{amsart}

\usepackage[margin=1in]{geometry} 
\usepackage[utf8]{inputenc} 
\usepackage[T1]{fontenc}    
\usepackage[english]{babel}
 \def\sloppy{%
  \tolerance 500%
  \emergencystretch 3em%
  \hfuzz .5pt
  \vfuzz\hfuzz}
\usepackage{graphicx,amssymb,subfigure,icomma}
\usepackage[labelfont=bf, width=\linewidth]{caption}
\usepackage{hyperref}  
\hypersetup{colorlinks=true,allcolors=black}
\usepackage{hypcap}   
\usepackage{bookmark} 
\usepackage{amsmath, amssymb, amscd, amsthm, amsfonts}
\usepackage{mathtools}
\usepackage{physics}
\usepackage{stackengine}
\usepackage{tikz}
\usepackage{xcolor}
\usepackage{fullpage}
\usepackage{bm}
\usepackage[all]{xy}
\usepackage[shortlabels]{enumitem}
\usepackage{multicol}
\usepackage{orcidlink}

\newtheorem{cor}{Corollary}[section]
\newtheorem{defn}[cor]{Definition}
\newtheorem{ex}[cor]{Example}
\newtheorem{lem}[cor]{Lemma}
\newtheorem{prop}[cor]{Proposition}
\newtheorem{rem}[cor]{Remark}
\newtheorem{thm}[cor]{Theorem}

\numberwithin{equation}{section}

\newcommand{\C}{\mathbb{C}}

\newcommand{\R}{\mathbb{R}}
\newcommand{\Q}{\mathbb{Q}}

\newcommand{\Z}{\mathbb{Z}}

\newcommand{\TT}{\mathbb{T}}

\newcommand{\m}{\mathrm{m}}
\newcommand{\Rh}{\mathcal{R}}
\newcommand{\K}{\mathcal{K}}
\newcommand{\nnum}{\nonumber}
\newcommand{\QQ}{\tilde{Q}}

\newcommand{\mm}{\mathcal{M}}

\newcommand{\Li}{\mathrm{Li}} 
 
\newcommand{\re}{\mathop{\mathrm{Re}}} 
\newcommand{\im}{\mathop{\mathrm{Im}}}

\title{Generalized Mahler measures of Laurent polynomials}
\author{Subham Roy}
\address{D\'epartement de math\'ematiques et de statistique, Universit\'e de Montr\'eal, CP 6128, succ. Centre-Ville. Montr\'eal, QC H3C 3J7, Canada}
\email{subham.roy@umontreal.ca}

\begin{document}


\maketitle

\begin{abstract}
   Following the work of Lal\'in and Mittal on the Mahler measure over arbitrary tori, we investigate the definition of the generalized Mahler measure for all Laurent polynomials in two variables when they do not vanish on the integration torus. We establish certain relations between the standard Mahler measure and the generalized Mahler measure of such polynomials. Later we focus our investigation on a tempered family of polynomials originally studied by Boyd, namely $Q_{r}(x, y) = x + \frac{1}{x} + y + \frac{1}{y} + r$ with $r \in \C,$ and apply our results to this family. For the $r = 4$ case, we explicitly calculate the generalized Mahler measure of $Q_4$ over any arbitrary torus in terms of special values of the Bloch--Wigner dilogarithm. Finally, we extend our results to the several variable settings.  
\end{abstract}

\section{Introduction}

The (logarithmic) Mahler measure of a non-zero rational function $P \in \C\left(x_1, \dots, x_n\right)^*$ is defined by 
\begin{equation}\label{defmm}
 \m\left(P\right)  = \m(P(x_1,\dots, x_n)):=\frac{1}{\left(2\pi i\right)^n}\int_{\mathbb{T}^n}\log|P\left(x_1, \dots, x_n\right)|\frac{dx_1}{x_1}\cdots \frac{dx_n}{x_n},
\end{equation}
where $\TT^n=\{\left(x_1, \dots, x_n\right)\in \C^* \times \C^* \times \cdots \times \C^* : |x_1|=\cdots=|x_n|=1\}$.

The first appearance of this quantity (for one variable polynomials) can be traced back to Lehmer's work \cite{L1} on Mersenne numbers, and its several variable form first appeared in the work of Mahler \cite{M1} regarding a simpler proof of the Gel'fond-Mahler inequality, and it was later named after him.

In the early $80$'s, Smyth \cite{Sm1} discovered the following remarkable identities: 
\begin{align*}
\m(x+y+1) =& \frac{3\sqrt{3}}{4\pi} L(\chi_{-3}, 2), \\ \m(1 + x + y + z) =& \frac{7}{2\pi^2}\zeta(3),
\end{align*} 
where $L(\chi_{-3}, 2)$ is the Dirichlet $L$-function of the quadratic character $\chi_{-3}$ of conductor $3,$ and  $\zeta(s)$ is the Riemann zeta function (for more details see \cite{Bo1}). These are two of the initial formulas for several variable cases.

Later the work of Boyd \cite{Bo2}, Deninger \cite{Den1}, Rodriguez-Villeags \cite{RV1} and others provided us with interesting connections among Mahler measure, higher regulators, and Be\u\i linson's conjectures. The conjectural formulas to support their work, such as \begin{equation*}
\m(P_k(x, y)) \stackrel{?}{=} r_k L'(E_{N(k)}, 0), \quad \quad r_{k} \in \Q,
\end{equation*} were eventually proved for certain polynomials, due to Rodriguez-Villegas \cite{RV1}, Rogers and Zudilin \cite{RZ1, LR07} et al. Here  $E_{N(k)}$ is an elliptic curve of conductor $N(k)$ associated to $P_k,$ and the question mark stands for a numerical formula that is true for at least 20 decimal places. (See the book of Brunault and Zudilin \cite{BZ20} for more details.)   

In a different direction, Cassaigne and Maillot \cite{CM1} generalized the formula found by Smyth to $\m(a + bx + cy)$ for arbitrary complex constants $a, b, \ \mbox{and} \ c:$

\begin{equation}\label{CM}
    \pi\m(ax + by + c) = \left\{\begin{array}{ll}
 \alpha \log |a| + \beta \log |b| + \gamma \log |c| + D\left(\frac{|a|}{|b|} e^{i\gamma}\right) & \mbox{if} \ \Delta \ \mbox{holds},  \\
 \log \max \{|a|, |b|, |c|\} & \mbox{if} \ \Delta \ \mbox{does not hold}, 
\end{array} \right.
\end{equation}
where $\Delta$ stands for the statement that $|a|, |b|, \ \mbox{and} \ |c|$ are the lengths of the sides of a planar triangle, and in that case, $\alpha, \beta, \ \mbox{and} \ \gamma$ are the angles opposite to the sides of the lengths $|a|, |b|$ and $|c|$ respectively (see Figure \ref{Fig1}). 

\begin{figure}[htb]
\centering
    \begin{tikzpicture}

\node (a1) at (0,1.5) {$\beta$};
\node (a2) at (-2,.3) {$\alpha$};
\node (a3) at (1,.3) {$\gamma$};
\node (a4) at (-.5,-.5) {$|b|$};
\node (a5) at (2, 1) {$|a|$};
\node (a6) at (-2, 1.5) {$|c|$};
\draw (0,2) -- (-3,0) -- (2,0) -- cycle;

\end{tikzpicture}
\caption{Condition $\Delta$ in Cassaigne and Maillot's formula}\label{Fig1}
\end{figure}
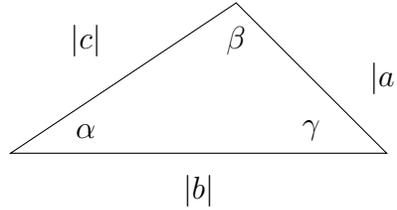

We also remark that the constant coefficient can be replaced by a variable without changing the Mahler measure, in the sense that $\m(ax + by + c) = \m(ax + by + cz).$ Additionally, it is immediate to see that Cassaigne and Maillot's result can also be interpreted as \begin{align*}
         \m(ax + by + cz) &= \frac{1}{\left(2\pi i\right)^3}\int_{\TT_{|a|, |b|, |c|}^3} \log|x + y + z|\frac{dx}{x}\frac{dy}{y}\frac{dz}{z},
     \end{align*} 
i.e. the standard Mahler measure of $ax + by + cz$ is same as the integral  of $\log|x + y + z|$ with respect to the Haar measure $\frac{d x}{x}\frac{dy}{y},$ and the integration domain is the torus $\TT_{|a|, |b|, |c|}^3,$ where \[\TT_{|a|, |b|, |c|}^3 = \{(x, y, z) \in \C^* \times \C^* \times \C^*: |x| = |a|, |y| = |b|, |z| = |c|\}.\] This different representation of $\m(ax + by +cz)$ makes \eqref{CM} a generalization of Smyth's result. This leads to the following definition.

\begin{defn}\label{def3}
The \textbf{generalized Mahler measure} of a non-zero rational function $P \in \C(x_1, \dots, x_n)^*$ is defined as \[ \m_{\mathfrak{a}}(P) = \m_{a_1, \dots, a_n} (P(x_1, \dots, x_n)):=\frac{1}{\left(2\pi i\right)^n}\int_{\TT_{\mathfrak{a}}^n} \log|P\left(x_1, \dots, x_n\right)|\frac{dx_1}{x_1}\cdots \frac{dx_n}{x_n},\] where $\mathfrak{a} = (a_1, \dots, a_n) \in (\R_{> 0})^n$ and \[\TT_{\mathfrak{a}}^n := \{(x_1, \dots, x_n) \in \C^* \times \C^* \times \cdots \times \C^* : |x_1| = a_1, \dots, |x_n| = a_n \}.\]
\end{defn}

Lal\'in and Mittal \cite{LM18} explored this definition over  $\TT_{q^2, q}^2$ and $\TT_{q, q}^2$ to obtain relations between some polynomials mentioned in Boyd's paper \cite{Bo2}, namely 
\[\begin{array}{ll}
 R_{-2}(x, y)&:=\displaystyle (1+x)(1+y)(x+y)+ 2xy,\\
 S_{2, -1}(x, y)&:=\displaystyle y^2 + 2xy - x^3 + x,
\end{array}\]
for some values of $q \in \R_{> 0}.$ They have simultaneously evaluated $\m_{q^2, q}(R_{-2})$ and $\m_{q, q}(S_{2, -1})$ in terms of $\log q$ and special values of $L$-functions when each of them does not vanish on the respective integration torus. In particular, they established a relation between the standard Mahler measure and the generalized Mahler measure. In this article, we provide a way to obtain such relations for all Laurent polynomials. Note that any Laurent polynomial can be expressed as \[Q_r(x, y) = r - Q(x, y) \in \C[x^{\pm}, y^{\pm}],\] where $r \in \C,$ and $Q$ has no constant term. Therefore, 

This project started with a specific family of Boyd's polynomials, namely \begin{equation}\label{firsteq*}
\left\{x + \frac{1}{x} + y + \frac{1}{y} +  t: t \in \C \right\}.
\end{equation} An extension of the methods in \cite{RV1} and \cite{Ber1} led us to an interesting fact: for an arbitrarily fixed $(a, b) \in \R_{>0}^2,$ there exists a large set of $t \in \C$ such that Mahler measure of these polynomials remains the same irrespective of deforming the integration torus from $\TT^2$ ($= \TT_{1, 1}^2$) to $\TT_{a, b}^2.$ In fact, we found that this method can be extended to all Laurent polynomials in $n$-variable (where $n \geq 2$) when they do not vanish on the integration torus.   

Let $P_k(x_1, \dots, x_n) \in \C[x_1^{\pm}, \dots, x_n^{\pm}]$ be a Laurent polynomial in $n$-variable such that \[P_k := P_k(x_1, \dots, x_n) = k - P(x_1, \dots, x_n),\] where $P$ has no constant term. Let $\TT_{\mathfrak{a}}^n$ be the integration torus in the definition of $\m_{\mathfrak{a}}(P_k),$ where $\mathfrak{a} = (a_1, \dots, a_n).$ 

For $a_1, \dots, a_n > 0,$ let $\K_{\mathfrak{a}}$ be the image of the map \[p:\TT_{\mathfrak{a}}^n \rightarrow \C \quad \mbox{defined by} \quad (x_1, \dots, x_n) \mapsto P(x_1, \dots, x_n).\] Let $\nu_{\mathfrak{a}, k}^j$ be the difference between the number of zeroes (counting multiplicities) of $P_{k}(a_1, \dots, a_{j-1}, x_j, a_{j+1}, \dots, a_n)$ inside the circle $\TT_{a_j}^1,$ denoted by $Z_{\mathfrak{a}, k}^j,$  and the order of the pole of $P_{k}(a_1, \dots, a_{j-1}, x_j, a_{j+1}, \dots, a_n)$ at $x_j = 0,$ denoted by $P_{\mathfrak{a}, k}^j.$ In other words, \[\nu_{\mathfrak{a}, k}^j = Z_{\mathfrak{a}, k}^j - P_{\mathfrak{a}, k}^j.\] Then, we have the following theorem. 

\begin{thm}\label{nvarthm}
 Let $\mathfrak{a} = (a_1, \dots, a_n) \in (\R_{>0})^n.$ Let $P_k(x_1, \dots, x_n) = k - P(x_1, \dots, x_n) \in \C[x_1^{\pm}, \dots, x_n^{\pm}],$ such that $P$ has no constant term. Denote $U_{\mathfrak{a}}$ the unbounded open connected component of $\C \setminus \K_{\mathfrak{a}}$ containing some neighbourhood of $k = \infty.$ Then, for $k \in U_{\mathfrak{a}} \cap U_{\mathfrak{1}},$ \begin{equation*}
\m_{\mathfrak{a}}(P_k) = \m(P_k) +  \sum_{j = 1}^n \nu_{\mathfrak{a}, k}^j \log a_j, 
\end{equation*} where $\nu_{\mathfrak{a}, k}^j$ is defined as above, and $\m_{\mathfrak{1}}(P_k) = \m(P_k).$ Moreover, for $k \in  U_{\mathfrak{a}} \cap U_{\mathfrak{1}}$ and $j = 1, \dots, n,$ $\nu_{\mathfrak{a}, k}^j$ only depends on $\mathfrak{a}.$ 
 \end{thm}  

 For $|k|$ large enough, the above relation between the standard Mahler measure and the generalized Mahler measure of $P_k$ can be obtained by first expanding $\log \left(1 - \frac{P}{k}\right)$ in a convergent series, and then integrating each term individually. We should mention here that, in order to obtain a convergent series expansion of the logarithm, the above procedure is restricted to a smaller subregion contained in the unbounded region of $\C \setminus \K_{\mathfrak{a}}.$ Theorem \ref{nvarthm} establishes this equality for a larger set, and since Mahler measure is the real part of an analytic function \cite{RV1} (in other words, since it is harmonic), this property of Mahler measure implies that the equality holds for all $k$ in the unbounded open connected component of $\C \setminus \K_{\mathfrak{a}}.$ In particular, we note that the region $U_{\mathfrak{a}} \cap U_{\mathfrak{1}}$ contains a neighborhood of $k = \infty,$ namely the region \[\left\{k \in \C : |k| > \max\left\{\max_{(x_1,\dots, x_n) \in \TT_{\mathfrak{a}}^n} |P\left(x_1,\dots, x_n\right)|, \max_{(x_1,\dots, x_n) \in \TT^n} |P\left(x_1,\dots, x_n\right)|\right\}\right\}.\] Indeed note that for all those $k,$ $\log \left(1 - \frac{P}{k}\right)$ is well defined and can be expanded in a convergent series, as mentioned above. Also, note that the region is therefore unbounded, and its complement is contained in \[\left\{k \in \C : |k| \leq \max \left\{\max_{(x_1,\dots, x_n) \in \TT_{\mathfrak{a}}^n} |P\left(x_1,\dots, x_n\right)|, \max_{(x_1,\dots, x_n) \in \TT^n} |P\left(x_1,\dots, x_n\right)|\right\}\right\},\] which is closed and bounded.

Let $Q(x, y)$ be a Laurent polynomial in $\C[x^{\pm}, y^{\pm}]$ with no constant term, and define the family of Laurent polynomials $\{Q_r(x, y): r \in \C\}$ associated to $Q$ as \[Q_r(x, y) = r - Q(x, y) \in \C[x^{\pm}, y^{\pm}].\] For $a, b > 0,$ let $\Rh_{a, b}$ be the image of the map \begin{equation}\label{Rh} 
q:\TT_{a, b}^2 \longrightarrow \C, \quad \mbox{defined by} \quad (x, y) \mapsto Q(x, y).
 \end{equation} 
Then, as a corollary to Theorem \ref{nvarthm}, we have the following theorem in $2$-variable.

\begin{thm}\label{mainthm}
 Let $a$ and $b$ be positive real numbers, and denote by $U_{a, b}$ the unbounded open connected component of $\C \setminus \Rh_{a, b}$ containing some neighborhood of $r = \infty.$ Then, for $r \in U_{a, b} \cap U_{1, 1},$ \begin{equation*}
\m_{a, b}(Q_r) = \m(Q_r) +  \nu_{a, b, r}^1 \log a + \nu_{a, b, r}^2 \log b, 
\end{equation*}  where $\nu_{a, b, r}^1$ is the difference between the number of zeroes (denoted by $Z_{a, b, r}^1$) and the number of poles (denoted by $P_{a, b, r}^1$)  of $Q_r(x, b)$ inside the circle $|x| = a,$  defined, by \begin{equation}\label{nudefn1}
    \nu_{a, b, r}^1 = Z_{a, b, r}^1 - P_{a, b, r}^1,
\end{equation} $\nu_{a, b, r}^2$ is the difference between the number of zeroes (denoted by $Z_{a, b, r}^2$) and the number of poles (denoted by $P_{a, b, r}^2$)  of $Q_r(a, y)$ inside the circle $|y| = b,$  defined, by \begin{equation}\label{nudefn}\nu_{a, b, r}^2 = Z_{a, b, r}^2 - P_{a, b, r}^2,\end{equation} and  $\m_{1, 1}(Q_r) = \m(Q_r).$ Moreover, for $r \in  U_{a, b} \cap U_{1, 1},$ $\nu_{a, b, r}^j$ does not depend on $r.$ 
 \end{thm}

A follow-up question can be posed regarding the values of $\m_{\mathfrak{a}}(P_k)$ when $k$ belongs to one of the bounded open connected components of  $\C \setminus \K_{\mathfrak{a}}.$ The next theorem answers this question when $\nu_{\mathfrak{a}, k}^j$ satisfies a particular condition. 

We introduce some necessary notations to state the next theorem. Multiplying $P_k$ with a suitable power of $x_{j},$ we can factorise $P_k$ in linear factors with coefficients in $\overline{\C(x_1, \dots, \widehat{x_j}, \dots, x_{n})}$ as \[P_{k}(x_1, \dots, x_n) = x_j^{-v_j}P_{F, k}^j(x_1, \dots, \widehat{x_j}, \dots, x_{n}) \prod_{l = 1}^{d_n} \left(x_j - X_{l, k, j}\left(x_1, \dots, \widehat{x_j}, \dots, x_{n}\right)\right),\] where $d_j$ is the degree of $P_k$ as a polynomial in $x_j,$ $X_{l, k, j}$ are algebraic functions of $(x_1, \dots, \widehat{x_j}, \dots, x_{n})$ for $l = 1, \dots, d_n,$ $P_{F, k}^j$ is the leading coefficient with respect to the variable $x_j,$ and $v_j$ is the largest power of $x_j^{-1}$ in $P_k.$ Let $P_{f, k}^j(x_1, \dots, \widehat{x_j}, \dots, x_{n})$ denote the "constant" coefficient with respect to the variable $x_j.$ Then \[P_{F, k}^j(x_1, \dots, \widehat{x_j}, \dots, x_{n}) \prod_{j = 1}^{d_n} X_{l, k, j}(x_1, \dots, \widehat{x_j}, \dots, x_{n}) = P_{f, k}^j(x_1, \dots, \widehat{x_j}, \dots, x_{n}).\] Then, we have the following theorem.
\begin{thm}\label{nvarconstant}
Let $\mathfrak{a} = (a_1, \dots, a_n) \in \R_{> 0}^n.$ Let $k_0 \in \C \setminus \K_{\mathfrak{a}}$ such that $k_0$ belongs to one of the bounded open connected components of $\C \setminus \K_{\mathfrak{a}}.$ We denote by $V_{\mathfrak{a}, k_0}$ the bounded open connected component containing $k_0.$  \begin{enumerate}
\item[(I)] For $j = 1, \dots, n,$ if all the roots of $P_{k_0}(a_1, \dots, a_{j-1}, x_j, a_{j+1}, \dots, a_n)$ lie entirely inside the circle $\TT_{a_j}^1,$ then, for all $k \in V_{\mathfrak{a}, k_0},$ \begin{equation*}
\m_{\mathfrak{a}}(P_k) = \nu_{\mathfrak{a}, k}^j \log a_j  + \m_{a_1, \dots, \widehat{a_j}, \dots, a_n}\left(P_{F, k}^j\right).
\end{equation*}

\medskip

\item[(II)] For $j = 1, \dots, n,$ if all the roots of $P_{k_0}(a_1, \dots, a_{j-1}, x_j, a_{j+1}, \dots, a_n)$ lie entirely outside the circle $\TT_{a_j}^1,$ then, for all $k \in V_{\mathfrak{a}, k_0},$ \begin{equation*}
\m_{\mathfrak{a}}(P_k) = \nu_{\mathfrak{a}, k}^j \log a_j  + \m_{a_1, \dots, \widehat{a_j}, \dots, a_n}\left(P_{f, k}^j\right).
\end{equation*} 
\end{enumerate}

\end{thm}

Similarly, for the $2$-variable case, $Q_{r}(x, y),$ when considered as a polynomial in $y$ (resp. $x$) of degree $d_y$ (resp. $d_x$) with coefficients in $\overline{\C(x)}$ (resp. $\overline{\C(y)}$), can be expressed as \begin{align*}
Q_r(x, y) =& (y)^{-v_2}\left(Q_{F, r}^y (x)(y)^{d_y} + Q_{f, r}^y(x) + \sum_{j=1}^{d_y-1}a_{j, r}^y(x)(y)^j\right)   \\ =& (x)^{-v_1}\left(Q_{F, r}^x (y)(x)^{d_x} + Q_{f, r}^x(y) + \sum_{j=1}^{d_x-1}a_{j, r}^x(y)(x)^j\right), 
\end{align*}  
where $v_1$ and $v_2$ denote the largest powers of $x^{-1}$ and $y^{-1}$ in $Q_{r}(x, y),$ respectively, and $Q_{F, r}^u$ and $Q_{f, r}^u$ are the respective leading and "constant" coefficient with respect to the variable $u,$ for $u =x \ \mbox{or} \ y.$ Then, again as a corollary to Theorem \ref{nvarconstant}, we have the following theorem. 

\begin{thm}\label{constant}
Let $a$ and $b$ be positive real numbers. Let $r_0 \in \C \setminus \Rh_{a, b}$ such that $r_0$ belongs to one of the bounded open connected components of $\C \setminus \Rh_{a, b}.$ We denote by $V_{a, b, r_0}$ the bounded open connected component containing $r_0.$  \begin{enumerate}[\textsl{\textbf{ (\roman*)}},ref=(\textbf{\roman*})]
\item \label{item:firstcond} If all the roots of $Q_{r_0}(a, y)$ either lie entirely inside the circle $\TT_{b}^1$ or lie entirely outside the circle $\TT_{b}^1,$ then, for all $r \in V_{a, b, r_0},$ 
\begin{equation*}
\m_{a, b}(Q_r) - \nu_{a, b, r}^2 \log b = \left\{\begin{array}{ll}
\m_{a}(Q_{F, r}^y(x)) & \mbox{when all roots of} \ Q_{r_0}(a, y) \ \mbox{lie inside} \ \TT_{b}^1, \\ \m_{a}(Q_{f, r}^y(x)) & \mbox{when all roots of} \ Q_{r_0}(a, y) \ \mbox{lie outside} \ \TT_{b}^1.
\end{array} \right.
\end{equation*}

\item \label{item:secndcond} If all the roots of $Q_{r_0}(x, b)$ either lie entirely inside the circle $\TT_{a}^1$ or lie entirely outside the circle $\TT_{a}^1$, then, for all $r \in V_{a, b, r_0},$ 
\begin{equation*}
\m_{a, b}(Q_r) - \nu_{a, b, r}^1 \log a = \left\{\begin{array}{ll}
\m_{b}(Q_{F, r}^x(y)) & \mbox{when all roots of} \ Q_{r_0}(x, b) \ \mbox{lie inside} \ \TT_{a}^1, \\ \m_{b}(Q_{f, r}^x(y)) & \mbox{when all roots of} \ Q_{r_0}(x, b) \ \mbox{lie outside} \ \TT_{a}^1.
\end{array} \right.
\end{equation*}
\end{enumerate}

\end{thm}

Using Theorems \ref{mainthm} and \ref{constant}, Cassaigne and Maillot's result in \eqref{CM} follows immediately when the "triangle condition" does not hold. In this case, $M_{c}(x, y) = c - x - y$ for $c \in \C.$ For $a, b \in \C^*,$ $\Rh_{|a|, |b|}$ for this family of polynomials is the "closed" annulus $\left\{z \in \C: |z| \in \left[\left||a| - |b|\right|, |a|+|b|\right]\right\}.$ Note that, when $c$ belongs to the unbounded component of $\C \setminus \Rh_{|a|, |b|},$ we have $\nu_{|a|, |b|, c}^j = 0.$ Then, Theorem \ref{mainthm} and harmonic properties of Mahler measure imply that, when $|c| > |a| + |b|,$ \[\m_{|a|, |b|}(M_c) = \m(|a|x + |b|y + c) = \log |c|.\] On the other hand, Theorem \ref{constant} implies that, for $|c| < \left||a| - |b|\right|,$ \[\m_{|a|, |b|}(M_c) = \log \max \{|a|, |b|\},\] since $\nu_{|a|, |b|, c}^1 = 1$ (resp. $\nu_{|a|, |b|, c}^2 = 1$) when $|a| > |b|$ (resp. $|b| > |a|$). The combination of both equalities leads to a restatement of \eqref{CM} when $\Delta$ does not hold. We should remark that the condition $\Delta$ in their result is equivalent to the condition $c \in \Rh_{|a|, |b|},$ i.e. $M_c$ vanishes on the integration torus. A more involved approach, using Theorems \ref{mainthm} and \ref{constant} on the family of polynomials \[\begin{array}{ll}
 R_{\alpha}^*(x, y)&:=\displaystyle \alpha - x^{-1} - y^{-1} - xy^{-1} - yx^{-1} - x - y, \quad \quad \alpha \in \C, \\
 S_{\beta, -1}^*(x, y)&:=\displaystyle \beta - x^{-1}y + x^2y^{-1} - y^{-1}, \quad \quad \beta \in \C,
\end{array}\] re-establishes the identities obtained in \cite{LM18} for $\alpha = -4$ and $\beta = 2.$ Note that the aforementioned result(s) involving the generalized Mahler measure of $R_{-2}^*$ (resp. $S_{2, -1}^*$) on the torus $\TT_{a, b}^2$ only depends on $b$ since the integration torus is $\TT_{b^2, b}^2$ (resp. $\TT_{b , b}^2$), i.e. $a$ is a function of $b$ here. Our theorems, along with the method of the \textit{Lagrange multiplier}, provide a larger set of pairs $(a, b) \in \R_{> 0}^2,$ such that similar types of identities obtained in \cite{LM18} hold even when $a$ is not a function of $b.$ An analogous result is exhibited in Section $5$ with a different family of polynomials investigated by Boyd \cite{Bo2}, namely the family given in \eqref{firsteq*}.

Due to the technical difficulties involving the study of the integration path in the definition of Mahler measure, it is challenging to evaluate $\m_{\mathfrak{a}}(P_k)$ explicitly for all $\mathfrak{a} \in \left(\R_{>0}\right)^n.$ In this regard, Theorems \ref{nvarthm} and \ref{nvarconstant} have a common feature: \textit{the Laurent polynomial in consideration does not vanish on the integration torus}. Since the method of proof is the same for Theorem \ref{nvarthm} (resp. Theorem \ref{nvarconstant}) and Theorem \ref{mainthm} (resp. Theorem \ref{constant}), we provide proofs of Theorems \ref{mainthm} and \ref{constant} in later sections, and outline follow-up arguments generalizing our method to derive Theorems \ref{nvarthm} and \ref{nvarconstant}. 

The next result considers a particular polynomial from our initial family of polynomials, namely \[Q_4(x, y) = x + \frac{1}{x} + y +\frac{1}{y} + 4.\] It removes the constraint of being non-zero on the integration torus, and evaluates the generalized Mahler measure of $Q_4(x, y)$ for all $a, b > 0.$ 

\begin{thm}\label{k=4*}
Let $a, b \in \R_{> 0},$ and define \[c = \sqrt{ab}, \quad d = \sqrt{\frac{b}{a}}, \quad \mbox{and} \ \mathcal{A}_{c, d} =  \frac{1 - d^2}{1 + d^2}\cdot \frac{1+c^2}{2c},\] such that $c$ and $d$ are both positive real numbers. Then, \begin{equation*}
\m_{a, b}(Q_4(x, y)) = \left\{\begin{array}{ll} \max\{\log c, -\log c\} + \max\{\log d, -\log d\} \quad &\mbox{if} \ \left|\mathcal{A}_{c, d}\right| \geq 1, \\ \\ \frac{2}{\pi}\left[ D(ice^{-i\mu}) + D(ice^{i\mu}) - \mu \log d + (\log c) \tan^{-1}\left(\frac{c - c^{-1}}{2\cos \mu}\right)\right] \quad &\mbox{if} \ \left|\mathcal{A}_{c, d}\right| < 1, \end{array}\right.
\end{equation*} where $\mu = \sin^{-1}\left(\mathcal{A}_{c, d} \right) \in \left(-\frac{\pi}{2}, \frac{\pi}{2}\right),$ and $D$ is the Bloch--Wigner dilogarithm defined, for $z \in \C,$ by \begin{equation*}
D\left(z\right)= \im\left(\Li_2\left(z\right)+ i\arg\left(1-z\right) \log|z|\right), \quad \quad \Li_2\left(z\right)=-\int_0^z\frac{\log\left(1-v\right)}{v}dv.
\end{equation*} 
\end{thm}

Under a certain change of variables, the polynomial above can be factored into two linear polynomials \cite{Bo2}. This simplification along with a direct approach involving a particular differential form and the Bloch--Wigner dilogarithm lead us to the explicit formula in the statement of Theorem \ref{k=4*}. 

This article is organised as follows. In Section $2,$ we recall some relations between Mahler measure, the Bloch--Wigner dilogarithm and a particular differential form appearing after a simplification of the definition in \eqref{defmm}. We end the section with a discussion on Mahler measure over an arbitrary torus. In Section $3,$ we discuss the proof of Theorem \ref{mainthm} and some auxiliary results required to complete the proof. We conclude the section with a brief argument generalizing our method to the several variables setting and proving Theorem \ref{nvarthm}. Section $4$ is completely dedicated to the proof of Theorem \ref{constant} and subsequently to the proof of Theorem \ref{nvarconstant} by a similar generalization. In Section $5,$ we discuss some applications of Theorems \ref{mainthm} and \ref{constant} to the family of polynomials in \eqref{firsteq*}. We then prove Theorem \ref{k=4*} in Section $6,$ where we use properties of the differential form and Bloch--Wigner dilogarithm mentioned in Section $2.$ We end this article with concluding remarks on possible directions to pursue going forward.

\section*{Acknowledgements}
I express my deepest gratitude to my Ph.D. supervisor, Matilde Lal\'in, for her invaluable assistance and support, and for sharing several ideas that have enriched this work. I am indebted to the anonymous referee for helpful remarks that greatly improved the exposition of this article. I am grateful to Andrew Granville and Christopher Deninger for their useful comments and enlightening discussions. I am also thankful to Ali Zahabi for helpful discussions on Ronkin functions and the generalized Mahler measure. Finally, I would like to express my gratitude to the Facult\'e des \'etudes sup\'erieurs et postdoctorales (bourses d'excellence) of the Universit\'e de Montr\'eal, the Institute des sciences math\'ematiques and the Centre de recherches math\'ematiques for their financial support.   
   
\section{Mahler measure and a differential form}

In this section, we briefly review some necessary background prior to proving the theorems in the following sections. 

\subsection{\textbf{Jensen's Formula}} \label{jensen}We recall a special case of Jensen's formula. Let $z_0 \in \C.$ Then \begin{align*}
\frac{1}{2 \pi i} \int_{\mathbb{T}^ 1}\log|z - z_0|\frac{dz}{z} = \left \{ \begin{array}{lr} \log |z_0 | & \quad |z_0 | \geq 1,\\\\ \quad 0 & \quad |z_0| \leq 1. \end{array} \right.
\end{align*}

\subsection{\textbf{Bloch--Wigner Dilogarithm}}
For $z \in \C,$ the Bloch--Wigner dilogarithm $D(z)$ is defined as
\begin{equation}\label{Bloch-Wigner}
D\left(z\right)= \im\left(\Li_2\left(z\right)+ i\arg\left(1-z\right) \log|z|\right),
\end{equation} where $\Li_2\left(z\right)=-\int_0^z\frac{\log\left(1-v\right)}{v}dv.$
In \cite{Za1}, Zagier shows that it can be extended continuously to $\C \cup \{ \infty \},$ with $D(\infty) = D(0) = D(1) = 0.$  In fact, it is real-analytic in $\C \setminus \{0, 1\}.$ 

For $z \in \C$, 
 \begin{align} \label{eq:relations}
 D\left(\bar{z}\right) = - D\left(z\right).
 \end{align}
This also implies that $D(r) = 0$ for all $r \in \R$  (for more details see \cite{Za1}). This property of the Bloch--Wigner dilogarithm frequently appears in the proof of Theorem \ref{k=4*}.

\subsection{\textbf{A differential form and its applications}}

Let $C$ be a curve over $\C$ which defines a compact Riemann surface, and  let $\C(C)$ be its field of fractions. For $f, g \in \C(C)^*,$ we define \begin{equation}\label{eq:eta}
\eta \left(f, g\right) := \log|f|  d \arg g - \log|g| d \arg f,
\end{equation} 
where $d\arg x$ is defined by $\Im(\frac{d x}{x})$. Note that, $\eta$ is a real $C^{\infty}$ differential $1$-form on $C \setminus S$, where $S$ contains all the zeroes and poles of $f$ and $g$. The following lemma consists of some useful properties of $\eta$ which are frequently used in later sections (see \cite{Gon1, Gon2}, \cite{DZ89} for more details).
\begin{lem} \label{lem:eta}
Let $f, g, h, v \in \C(C)^*$ and $a, b \in \C^*.$ Then we have 
\begin{enumerate}
\item \label{anttsym} $\eta(f, g) = -\eta(g, f),$ i.e. $\eta$ is anti-symmetric,
\item \label{multpl}$\eta(fg, hv) = \eta(f, h) + \eta(g, h) + \eta(f, v) + \eta(g, v),$ 
\item \label{constnt} $\eta(a, b) = 0,$
\item\label{closed}$\eta$ is a closed differential form,
\item for $x, 1-x \in \C(C)^*,$ \begin{equation}  \label{eta(x,1-x)}
 \eta\left(x, 1-x\right) = d D\left(x\right).
 \end{equation}
\end{enumerate}
\end{lem}

Let $P(x, y)$ be a Laurent polynomial in two variables. Multiplying $P(x,y)$ by a suitable power of $y,$ we can always assume that $P(x, y) \in \C[x^{\pm}, y]$ is a polynomial of degree $d$ in $y,$ where $d > 0$. Then $P(x,y)$ has the following factorization over $\overline{\C(x)}:$ \[P\left(x, y\right)=P^*\left(x\right)\left(y-y_1\left(x\right)\right)\left(y-y_2\left(x\right)\right) \cdots \left(y-y_d\left(x\right)\right),\]
where $P^*\left(x\right) \in \C[x]$ and $y_j := y_j\left(x\right)$ are algebraic functions of $x$ for $j = 1, 2, \dots, d.$ 

Applying Jensen's formula with respect to the variable $y$ in the standard Mahler measure formula for $P(x, y),$ we obtain 
\begin{align} 
\m\left(P (x, y)\right)-\m\left(P^* (x)\right) =& \frac{1}{\left(2 \pi i\right)^2} \int_{\TT^2}\log |P\left(x, y\right)|\frac{dx}{x}\frac{dy}{y}-\m\left(P^* (x)\right) \nnum \\
=& \frac{1}{2\pi i}\left(\sum_{j = 1}^{d} \int_{|x|=1, |y_j\left(x\right)|\geq 1} \log|y_j\left(x\right)|\frac{dx}{x}\right) \nnum \\
=& - \frac{1}{2 \pi} \sum_{j = 1}^{d} \int_{|x|=1, |y_j\left(x\right)|\geq 1} \eta \left(x, y_j \right), \label{stndrd mm}
\end{align}
where $\eta$ is defined by \eqref{eq:eta}, and $\eta(x, y_j) = i \log |y_j (x)| \frac{d x}{x},$ which immediately follows from the facts that $\log |x| = \log 1 = 0$ and $\frac{d x}{x} = d (\log |x| + i \arg x).$ Here we consider $\arg (x) \in [-\pi, \pi).$ 

Therefore, if $\eta$ can be decomposed as 
 \begin{equation} \label{eq:decompose}
 \eta\left(x, y_j\right) = \sum_{k} a_{j_k} \eta\left(z_{j_k}, 1 - z_{j_k} \right) = \sum_{k} a_{j_k} d D(z_{j_k}),
 \end{equation} 
where $z_{j_k}, \left(1 - z_{j_k}\right) \in \C(C)^*$ are algebraic functions of $x,$ and the sum is finite, then \eqref{stndrd mm} can be restated in terms of Bloch--Wigner dilogarithm:  \begin{equation*}
 \m \left(P\left(x, y\right)\right) - \m \left(P^*\left(x\right)\right) = -\frac{1}{2\pi} \sum_{j=1}^{d} \sum_{k} a_{j_k} D\left(z_{j_k}\right)|_{ \partial  \{ |x| =1, |y_j \left(x\right) \geq 1 \}},
\end{equation*} 
where $\partial \{ |x| = 1, |y_j| \geq 1 \}$ is the set of boundary points of $\{ |x| = 1, |y_j| \geq 1 \}.$

  \begin{rem} \label{remark2}
As mentioned in \cite{V2}, we may have some extra terms of the form  $\eta \left(c, z\right)$ in \eqref{eq:decompose}, where $c$ is a constant complex number and $z$ is some algebraic function. In that case, we can still reach a closed formula by integrating $\eta \left(c, z\right)$ directly (i.e. by integrating $\log |c| d \arg z$). Also, if $\nu$ is a constant such that $|\nu| = 1,$ then $\eta \left(\nu, z\right) = \log |\nu| d \arg z = 0$. 
 \end{rem}

 \subsection{\textbf{Arbitrary Tori, Mahler measure, and $\eta$}}\label{ATaMM}

For a Laurent polynomial $P(x, y),$ we analyze the Mahler measure of $P$ over an arbitrary torus $\TT_{a, b}^2.$ The following brief description is essentially reproducing the analysis in Section $3$ of \cite{LM18}. For simplicity, we take $d = 2$, where $d$ is the degree of $y$ in $P(x, y)$ once $P$ is multiplied by a suitable power of $y$ to remove any negative power of $y.$  

Let $x=ax'$ and $y=by'.$  Then we have, for $P^*(x) \in \C[x],$
\begin{align}
\m_{a, b}\left(P (x, y)\right)-\m_{a, b}\left(P^*(x)\right)
=&\frac{1}{\left(2\pi i\right)^2} \iint_{|x'|=|y'|=1} \log|P\left(ax', by'\right)|\frac{dx'}{x'}\frac{d y'}{y'}-\m_{a, b}\left(P^*(x)\right)\nnum \\   =& 2\log b + \frac{1}{2\pi i}\left(\sum_{j = 1}^{2} \int_{|x'|=1, |y'_j|\geq 1} \log|y'_j|\frac{dx'}{x'}\right), \nnum \\ =& 2\log b -\frac{1}{2\pi} \sum_{j =1}^2 \int_{|x|=a, |y_j|\geq b} \eta\left(x/a, y_j/b\right).  \label{ATMM}
\end{align}

where $y_j = y_j(x) = by'_j$ are algebraic functions of $x$ for $j = 1, 2,$ and \[\eta\left(x/a, y_j/b\right) = \eta(x', y'_j) = i \log|y'_j| \frac{d x'}{x'},\] for $j =1, 2,$ and the penultimate equality follows from Jensen's formula. Further simplification of the terms involving $y_j$'s using \eqref{multpl} of Lemma \ref{lem:eta} implies 
\begin{equation*}
\m_{a, b}\left(P (x, y)\right)-\m_{a, b}\left(P^*(x)\right)
=2\log b  -\frac{1}{2\pi} \sum_{j=1}^2 \int_{|x|=a, |y_j|\geq b} \left[\eta\left(x, y_j\right)-\eta\left(a, y_i\right)-\eta\left(x, b\right)\right].
\end{equation*}

If $\{|x|=a, |y_j|\geq b\}$ is a closed path, then the integral \[\frac{1}{2\pi} \sum_{j=1}^2 \int_{|x|=a, |y_j|\geq b} \eta\left(x/a, y_j/b\right)\] can be evaluated using Stokes' theorem (see Deninger \cite{Den1}). In addition, if $\{|x|=a, |y_i|\geq b\}$ is a closed path, the term  
\begin{equation*}
\frac{1}{2\pi}\int_{|x|=a, |y_j|\geq b} \eta\left(a, y_i\right) = \frac{\log a}{2\pi} \int_{|x|=a, |y_j|\geq b} d \arg y_j
\end{equation*}
becomes a multiple of $\log a.
$
If we have a genus $0$ curve (such as $C_4 : Q_4(x, y) = 0$) then, instead of proceeding as in the direction above, we may be able to use \eqref{eta(x,1-x)} to relate the Bloch--Wigner dilogarithm and $\eta$ for evaluating the Mahler measure. The evaluation is much simpler in this case as we will see in the proof of Theorem \ref{k=4*}. 

\section{Proof of Theorem \ref{mainthm}}

We are now ready to prove Theorem \ref{mainthm} for $Q_{r}(x, y) = r - Q(x, y) \in \C[x^{\pm}, y^{\pm}],$ where $Q(x, y)$ has no constant term. In this section, we use the notation \[\m_{a, b}(Q_r(x, y)) = \m_{a, b}(Q_r) = \m_{a, b}(r) \quad \quad \quad \mbox{for} \ r \in \C,\] for simplicity. Our approach is inspired from the methods of Rodriguez-Villegas \cite{RV1} and Bertin \cite{Ber1}. We first show that the required equality between Mahler measures holds for a smaller unbounded region of $\C \setminus \Rh_{a, b},$ and then we argue using properties of harmonic functions that it can be extended to the desired region stated in Theorem \ref{mainthm}. 

The following lemma formulates the invariance of $\m_{a, b}(r)$ under certain changes of variables.

\begin{lem}\label{Lem1}
Let $a, b$ be positive real numbers. Define $f_r(a, b) := \m_{a, b}(r).$ Then $f_r$ satisfies the following identities: \[f_r(a, b) = f_r(b, a) = f_r\left(\dfrac{1}{a}, b\right) = f_r \left(\dfrac{1}{a}, \dfrac{1}{b}\right).\]
\end{lem}

\begin{proof}[\textbf{Proof of Lemma \ref{Lem1}}]
Let $a, b > 0.$ For $\QQ_r (x, y) = Q_r(ax, by),$ the generalized Mahler measure of $Q_r$ satisfies the identity \[\m_{a, b}(Q_r(x, y)) = \m(\QQ_r(x, y)) = \m(\QQ_r).\] The change of variables \[(x, y) \rightarrow (y, x), \quad (x, y) \rightarrow (x^{-1}, y), \quad (x, y) \rightarrow (x^{-1}, y^{-1}),\] fix $\m(\QQ_r).$ Since $\m_{a, b}(r) = \m(\QQ_r),$ we have the required identities involving $f_r(a, b)=\m_{a, b}(r).$
\end{proof}

In view of Lemma \ref{Lem1}, we may assume without loss of generality that $a > b > 1.$

Our main aim is to study $\m_{a, b}(r)$ in terms of the complex parameter $r.$ Recall that $\Rh_{a, b}$ is the set of all $r \in \C$ such that $Q_r(x, y)$ vanishes on $\TT_{a, b}^2.$ Before proceeding to prove the theorem, we state a proposition explaining the following: \begin{itemize} 
  \item the behaviour of the roots of $Q_r(x, y)$ for each $x \in \TT_{a}^1;$ in particular, the number of roots inside the unit circle $\TT_{b}^1,$
 
 \item the behaviour of the roots of $Q_r(x, y)$ for each $y \in \TT_{b}^1;$ in particular, the number of roots inside the unit circle $\TT_{a}^1.$
 \end{itemize} 
This proposition, in particular, provides us with the quantities $\nu_{a, b, r}^2$ and $\nu_{a, b, r}^1$ in the statement of Theorem \ref{mainthm}. Since the above two cases are analogous, we consider just the first case. 

For $w \in \TT_{a}^1,$ let $\varrho_{a, b, r}^2(w)$ denote the number of roots of $Q_r(w, y)$ lying inside the circle $\TT_{b}^1.$ In particular, following the discussion preceding the definition in \eqref{nudefn}, we have, for $w \in \TT_{a}^1,$ \begin{equation}\label{varrhodefn}\varrho_{a, b, r}^2(w) = Z_{w, b, r}^2 \quad \mbox{and} \quad \varrho_{a, b, r}^2 (a) = Z_{a, b, r}^2 = \nu_{a, b, r}^2 + P_{a, b, r}^2 = \nu_{a, b, r}^2 + v_2,
\end{equation} where  $Z_{w, b, r}^2$ is the number of zeros (counting multiplicities)  of $Q_{r}(w, y)$ inside the circle $\TT_{b}^1$, $P_{w, b, r}^2$ is the order of the pole of $Q_r(w, y)$ at $y =0,$ and $v_2$ is the largest power of $y^{-1}$ in $Q_{r}(x, y).$ Then we have the following proposition.

\begin{prop}\label{allin1}
Let $r \in \C \setminus \Rh_{a, b}.$ Then $\varrho_{a, b, r}^2(x)$ is constant for all $x \in \TT_{a}^1.$  
\end{prop}

Before proceeding with the proof, we first consider the \textit{resultant} of the polynomial $Q_r$ with respect to $y.$ 

 Recall that \begin{equation*}
Q_r(x, y) = y^{-v_2}Q_{F, r}^y(x)\prod_{j=1}^{d_y} (y - y_{j, r}(x)),
\end{equation*}
 where $y_{j, r}(x)$ are algebraic functions in $x,$ and $v_2$ is as defined above. 

 Here and in what follows for the rest of this section, we denote $Q_{F, r}(x) := Q_{F, r}^y(x), d:= d_y.$  

Let $D_r(x)$ denote the \textit{resultant} of $Q_{r}(x, y)$ and $\frac{\partial}{\partial y}Q_{r}(x, y)$ with respect to $y.$

Then the algebraic solutions $y_{j, r}$ are holomorphic in some neighbourhood of $x$ for any $x \in \C \setminus S_r,$ where \begin{equation}\label{criticalpoint}
S_r = \{z \in \C: Q_{F, r}(z)D_r(z) = 0\}
\end{equation} is a finite subset of $\C.$ 

Let $\textbf{y}_r(x)$ be the $d$-valued global analytic function, with $d$-branches $y_{1, r}, \dots y_{d, r},$ such that $Q_{r}(x, \textbf{y}_r(x)) = 0.$ Then $S_r$ is called the set of \textit{critical points} of $\textbf{y}_r(x).$ If $x'$ is a critical point of $\textbf{y}_r(x),$ then $x'$ is either an algebraic branch point or a pole (for more details see \cite{LAlf}). \begin{enumerate}[(\textbf{\arabic*}),ref= (\textbf{\arabic*})]
\item \label{item:critical1} If $x' \in S_r$ is an algebraic branch point, i.e. when $D_{r}(x') = 0,$ then, in a sufficiently small neighbourhood $U_{x'}$ of $x'$ (which does not contain any other critical points), the multi-set $\{y_{1, r}, \dots y_{d, r}\}$ can be decomposed into a number of non-intersecting cycles \[\{f_1(x), \dots, f_{k_1}(x)\}, \dots, \{f_{k_1 + \dots + k_{t-1}+1}(x), \dots, f_{k_1 + \dots + k_t}(x)\},\] such that $\sum_{n= 1}^t k_n = d,$ and $f_{j}(x) = y_{l, r}(x)$ for some $j, l \in \{1, \dots, d\}.$ The elements of the first cycle can be represented as convergent Puiseux series of the local parameter $\tau = (x - x')^{1/k_1}$ in a small enough neighbourhood of $\tau = 0.$ The elements of the rest of the cycles follow analogous convergent series representations. Therefore, a single turn around $x'$ in a circle $C' \subset U_{x'}$ converts the Puiseux series of elements in one cycle into each other in a cyclic order, i.e. $f_{1} \rightarrow f_{2} \rightarrow \cdots \rightarrow f_{k_1} \rightarrow f_{1}$ etc.

\item \label{item:critical2} If $x' \in S_r$ is a pole, that is when $Q_{F, r}(x') = 0,$ then, substituting $y$ with $yQ_{F, r}(x),$ we return to the first case where the local parameter of the convergent series is $\tau = 1/x.$ 
\end{enumerate}

Recall that, for $w \in \TT_{a}^1,$ $\varrho_{a, b, r}^2(w)$ denote the number of roots of $Q_r(w, y)$ lying inside the circle $\TT_{b}^1.$ We are now ready to prove Proposition \ref{allin1}.

\begin{proof}[\textbf{Proof of Proposition \ref{allin1}}]
First fix an arbitrary $r \in \C \setminus \Rh_{a, b}.$ Note that $\varrho_{a, b, r}^2$ defines a function from $\TT_{a}^1$ to $\Z$ via the map $x \mapsto \varrho_{a, b, r}^2(x),$ where $\Z$ is equipped with discrete topology. 
 
If $x_0 \in \TT_{a}^1$ is not a critical point of $\textbf{y}_r,$ i.e. $x_0 \notin S_r,$ where $S_r$ is given in \eqref{criticalpoint}, then, for all $j = 1, \dots, d,$ $y_{j, r}$ is holomorphic in a sufficiently small neighbourhood $U_{x_0}$ of $x_0$ which does not contain any critical point. Therefore, $|y_{j, r}(x)|$ is continuous in $U_{x_0}.$ Since $Q_{r}$ does not vanish on $\TT_{a, b}^2,$ we have $|y_{j, r}(x)| \neq b$ for all $x \in U_{x_0} \cap \TT_{a}^1,$ $j.$ Therefore, if, for any $l = 1, \dots, d,$ $|y_{l, r}(x_0)| < 1$ (resp. $|y_{l, r}(x_0)|  > 1$), then, for all $x \in U_{x_0 \cap \TT_{a}^1},$ $|y_{l, r}(x)| < 1$ (resp. $|y_{l, r}(x_0)| > 1$). In other words, $\varrho_{a, b, r}^2(x)$ is constant for all $x \in U_{x_0} \cap \TT_{a}^1.$ In particular, $\varrho_{a, b, r}^2$ is continuous at $x_0.$ 
 
 If $x_1 \in \TT_{a}^1 \cap S_r,$ then there exists a sufficiently small neighbourhood $U_{x_1}$ of $x_1$ which does not contain any critical point except $x_1.$ Then the convergent Puiseux series expansions of $y_{1, r}, \dots, y_{d, r}$ in $U_{x_1}$ imply that, for all $j,$ $|y_{j, r}|$ is continuous in $U_{x_1},$ and this brings us to the previous case. From properties \ref{item:critical1}, \ref{item:critical2} and the above discussion, we conclude that, in the neighbourhood $U_{x_1}$ of $x_1,$ $\varrho_{a, b, r}^2$ is constant. This implies that $\varrho_{a, b, r}^2$ is continuous at $x_1.$ 
 
 We now have a continuous function $\varrho_{a, b, r}^2$ from a connected set $\TT_{a}^1$ to a discrete set $\Z.$ Since only connected subsets of $\Z$ are singletons, we derive that $\varrho_{a, b, r}^2$ is constant in $\TT_{a}^1,$ and thus completing the proof of the statement. 
 \end{proof}

\subsection{Proof of Theorem \ref{mainthm}}\label{pfmainthm}

Proposition \ref{allin1} implies that, for all $x \in \TT_{a}^1,$ $\varrho_{a, b, r}^2 (x) = \nu_{a, b, r}^2 + v_2.$ Moreover, \eqref{varrhodefn} implies that the constant is $\nu_{a, b, r}^2 + v_2,$ where $v_2$ is the largest power of $y^{-1}$ in $Q_{r}(x, y).$ In particular, we have $\nu_{w, b, r}^2 = \nu_{a, b, r}^2$ for all $w \in \TT_{a}^1,$ where $\nu_{w, b, r}^2$ is given in \eqref{nudefn}. Next, we derive Theorem \ref{mainthm} using Proposition \ref{allin1}.

\begin{proof}[\textbf{Proof}]

For $a$ and $b$ positive real numbers, the torus $\TT_{a, b}^2$ is defined as the set $\{(x, y) \in \left(\C^*\right)^2: |x| = a, |y| = b\}.$ By construction, $\TT_{a, b}^2$ is compact. Since the map in \eqref{Rh}, namely \[q:\TT_{a, b}^2 \longrightarrow \C, \quad \mbox{defined by} \quad (x, y) \mapsto Q(x, y),\] is continuous, the image of $q$ is compact. That is, $q(\TT_{a, b}^2) = \Rh_{a, b}$ is compact, and therefore closed and bounded in $\C.$ In other words, $\max_{r \in \Rh_{a, b}} |r|$ exists. We denote \[R_{a, b} := \max_{r \in \Rh_{a, b}} |r|, \quad \mbox{and} \quad R_{a, b, 1, 1} := \max\{R_{a, b}, R_{1, 1}\}.\]

Following a construction in \cite{RV1}, we define \[\tilde{\m}_{a, b}(r) = \log r - \sum_{n \geq 0}\frac{a_{n, a, b}}{n}r^{-n}, \quad |r| > R_{a, b, 1, 1}, r \notin (-\infty, 0],\] where $\log$ denotes the principal branch of the logarithm, and $a_{n, a, b}$ is defined as follows: \begin{align*}
a_{n, a, b} = \left[\frac{1}{(2\pi i)^2}\int_{\TT_{a, b}^2} \frac{dx dy}{xy(1-r^{-1}Q(x, y))}\right]_n = \frac{1}{(2\pi i)^2}\int_{\TT_{a, b}^2} Q(x, y)^n\frac{dx}{x}\frac{dy}{y}, 
\end{align*}

Here $[T(s)]_n$ denotes the coefficient of $s^{-n}$ in the series $T(s).$ It is immediate to see that $\tilde{\m}_{a, b}$ is holomorphic in the region defined by $|r| > R_{a, b, 1, 1}$ and $r \notin (-\infty, 0]$. Also, \[\Re(\tilde{\m}_{a, b}(r)) = \m_{a, b}(r), \quad |r| > R_{a, b, 1, 1}.\] 

We now claim that, for $|r| > R_{a, b, 1, 1},$ \[\frac{d\tilde{\m}_{a, b}}{dr} = \frac{d\tilde{\m}_{1, 1}}{dr}.\] 
In order to prove our claim, it is enough to show that $a_{n, a, b} = a_{n, 1, 1}$ for all $n.$ The above construction of the coefficients and the integral expression of these terms in \cite{RV1} yield that \begin{align*}
a_{n, a, b} &= \frac{1}{(2\pi i)^2}\int_{\TT_{a, b}^2} Q(x, y)^n\frac{dx}{x}\frac{dy}{y} = \frac{1}{(2\pi i)^2}\int_{\TT^2} Q(ax', by')^n\frac{dx'}{x'}\frac{dy'}{y'} \\ &= \left[Q(ax', by')^n\right]_0 = \left[Q(x', y')^n\right]_0 \\ &= \frac{1}{(2\pi i)^2}\int_{\TT^2} Q(x', y')^n\frac{dx'}{x'}\frac{dy'}{y'} \\ &= a_{n, 1, 1}.
\end{align*} The equality $\left[Q(ax', by')^n\right]_0 =  \left[Q(x', y')^n\right]_0$ follows from the fact that the constant term gathers the terms with degree $0,$ which are invariant under the multiplications of $x$ and $y$ by $a$ and $b,$ respectively. This concludes the proof.

Due to the above identity, we can denote the coefficients as $a_n:= a_{n, a, b} = a_{n, 1, 1}$ for the rest of the argument. From the definition of $\tilde{\m}_{a, b},$ it follows that \begin{equation}
\frac{d\tilde{\m}_{a, b}}{dr} = \frac{1}{(2\pi i)^2}\int_{\TT_{a, b}^2}\frac{1}{r - Q(x, y)}\frac{dx}{x}\frac{dy}{y}, \quad \quad |r| > R_{a, b, 1, 1}, \label{der1}
\end{equation}
where we include the region $r \in (-\infty, 0] \cap \{|r| > R_{a, b, 1, 1}\}$ by continuity. We need to show that $\frac{d\tilde{\m}_{a, b}}{dr}$ is in fact holomorphic in $|r| > R_{a, b, 1, 1}.$ 

For $r \in \C,$ define \begin{equation}\label{mathcalF}
\mathcal{F}_{a, b}(r) := \frac{1}{(2\pi i)^2}\int_{\TT_{a, b}^2}\frac{1}{r - Q(x, y)}\frac{dx}{x}\frac{dy}{y}.
\end{equation}

Note that the integrand \[\left.\frac{1}{r - Q(x, y)}\right|_{(x, y) \in \TT_{a, b}^2}\] is holomorphic in $r$ when $|r| > R_{a, b, 1, 1}.$ In fact, we will now show that $\mathcal{F}_{a, b}(r)$ is holomorphic as well on $|r| > R_{a, b, 1, 1}.$  The integrand, as well as the integral in \eqref{der1}, are bounded on $\TT_{a, b}^2$. This implies that $\frac{d^j\mathcal{F}_{a, b}}{dr^j}$ exists and is holomorphic for $j = 1$ (and therefore for all $j \geq 1$). Hence, $\mathcal{F}_{a, b}(r)$ is holomorphic in $|r| > R_{a, b, 1, 1}.$ 

Recall that we have, for $|r| > R_{a, b, 1, 1},$ \[\frac{d\tilde{\m}_{a, b}(r)}{dr} = \frac{d\tilde{\m}_{1, 1}(r)}{dr},\] and all the quantities are holomorphic in the mentioned region. Integrating both sides with respect to $r,$ we get \[\tilde{\m}_{a, b}(r) = \tilde{\m}_{1, 1}(r) + \tilde{f}(a, b), \quad \quad \mbox{for} \ |r| > R_{a, b, 1, 1},\] where $\tilde{f}(a, b)$ is the integration constant which only depends on $a$ and $b.$ Taking the real part of both sides yields \begin{equation}\label{2var2}
\m_{a, b}(r) = \m_{1, 1}(r) + f(a, b), \quad \quad \mbox{for} \ |r| > R_{a, b, 1, 1},  
\end{equation}
where $\re(\tilde{f}(a, b)) = f(a, b).$

Notice that $\m_{a, b}(r)$ is harmonic on $U_{a, b},$ the unbounded component of $\C \setminus \Rh_{a, b}$ which contains $\{|r| > R_{a, b}\},$ and $\m_{1, 1}(r) + f(a, b)$ is also harmonic on $U_{1, 1},$ since $f(a, b)$ is constant for $a, b$ fixed. The equality \eqref{2var2} implies that $\m_{a, b}(r)$ and $\m(r) + f(a, b)$ coincide in the open neighbourhood $|r| > R_{a, b, 1, 1}.$ Therefore, they must be equal in $U_{a, b} \cap U_{1, 1},$ that is \begin{align}
\Re(\tilde{\m}_{a, b}(r)) &= \m_{a, b}(r) = \m(r) + f(a, b), \quad \quad \quad \mbox{for} \ r \in O_{a, b} := U_{a, b} \cap U_{1, 1} \label{1eq} 
\end{align}

We now proceed to evaluate $f(a, b)$ in terms of $a, b.$ Since $\Rh_{a, b}$ is compact for $a, b > 0,$ it is bounded for such $a, b.$ Let $0< \delta <1$ such that $a, b > \delta.$ Let $\mm_{a, b}$ be the subset of $\R_{>0}^2$ defined by \[\mm_{a, b} = [a-\delta, a+\delta] \times [b-\delta, b + \delta].\] Note that $(a, b) \in \mm_{a, b}.$ Since $\mm_{a, b}$ is compact, and the map $(\alpha, \beta) \mapsto R_{\alpha, \beta}$ is continuous for all $(\alpha, \beta)$ in $\mm_{a, b},$ we conclude that the subset $\{R_{\alpha, \beta}: (\alpha, \beta) \in \mm_{a, b}\}$ is compact in $\R_{> 0}.$ Then $\tilde{R}_{a, b} := \max_{(\alpha, \beta) \in \mm_{a, b}} R_{\alpha, \beta}$ exists, and is finite. Now choose an $R \in \R_{> 0}$ such that \[R > \tilde{R}_{a, b} + R_{1, 1}.\] The choice of $R$ implies that, for $(\alpha, \beta) \in \mm_{a, b},$ $\tilde{\m}_{\alpha, \beta}(R)$ is holomorphic, and \eqref{2var2} yields 
\begin{equation}\label{new1}
\m_{\alpha, \beta}(R) = \m_{1, 1}(R) + f(\alpha, \beta).
\end{equation} 

Let $A_{a, b, \delta} \subset \C^2$ be the poly-annulus $A_{a, b, \delta} = A_{a, \delta} \times A_{b, \delta},$ where $A_{a, \delta} = \{z \in \C: a - \delta < |z| < a + \delta\}$ and $A_{b, \delta} = \{z \in \C: b - \delta < |z| < b + \delta\}.$ Note that $\TT_{a, b}^2 \subset A_{a, b, \delta}.$ Since \[Q_R(x, y) \in \C \setminus (-\infty, 0] \quad \quad \mbox{for} \ (x, y) \in A_{a, b, \delta},\] $\log(Q_R(x, y))$ is holomorphic in $A_{a, b, \delta},$ where $\log$ is the principal branch of logarithm. Let $\tilde{W}_{a, b}$ denote the set of all $(\alpha, \beta) \in \mm_{a, b}$  such that \[\tilde{\m}_{\alpha, \beta}(R) = \frac{1}{(2\pi i)^2}\int_{\TT_{\alpha, \beta}^2} \log(Q_R(x, y))\frac{dx}{x}\frac{dy}{y}.\] Note that $\tilde{W}_{a, b}$ is an open subset of $\mm_{a, b},$ and it also contains $(a, b).$ 

Next we compute the functions $\alpha \frac{\partial \tilde{\m}_{\alpha, \beta}(R)}{\partial \alpha}$ and $\beta \frac{\partial \tilde{\m}_{\alpha, \beta}(R)}{\partial \beta}.$ We only show here the computation of $\alpha \frac{\partial \tilde{\m}_{\alpha, \beta}(R)}{\partial \alpha}$ when $\alpha$ belongs to  an open subinterval of $(a - \delta, a + \delta)$ containing $a,$ since the other case is analogous. 

Note that $\tilde{\m}_{\alpha, \beta}(R)$ and $\log(Q_R)$ are well-defined and finite-valued on $\tilde{W}_{a, b}$ and $A_{a, b, \delta},$ respectively. Therefore, we can consider their partial derivatives with respect to $\alpha,$ and obtain \begin{align}
\alpha \frac{\partial \tilde{\m}_{\alpha, \beta}(R)}{\partial \alpha} &= \alpha \frac{ \partial}{\partial \alpha}\left(\frac{1}{(2\pi i)^2}\int_{\TT_{\alpha, \beta}^2} \log(Q_R(x, y))\frac{dx}{x}\frac{dy}{y}\right) \nnum \\ &= \frac{1}{(2\pi i)^2}\int_{\TT_{\alpha, \beta}^2} \alpha \frac{\partial \log(Q_R(x, y))}{\partial x}\frac{\partial x}{\partial \alpha}\frac{dx}{x}\frac{dy}{y} \nnum \\ &= \frac{1}{(2\pi i)^2}\int_{\TT_{\alpha, \beta}^2} x \frac{\partial_x Q_R(x, y)}{Q_R(x, y)}\frac{dx}{x}\frac{dy}{y} \nnum \\ &= \frac{1}{(2\pi i)^2} \int_{|y| = \beta}\left(\int_{|x| = \alpha}\frac{\partial_x Q_R(x, y)}{Q_R(x, y)} dx \right) \frac{dy}{y} , \label{partial1}
\end{align}  where $\partial_x = \frac{\partial}{\partial x},$ and the penultimate equality follows from the facts that $x = \alpha e^{i \theta}$ and $\theta$ does not depend on $\alpha.$ For a fixed $y_0$ such that $|y_0| = \beta,$ the integrand \[\int_{|x| = \alpha}\frac{\partial_x Q_R(x, y_0)}{Q_R(x, y_0)} dx = Z_{\alpha, y_0, R}^1 - P_{\alpha, y_0, R}^1\] is an integer, where $Z_{\alpha, y_0, R}^1$ denotes the number of zeros (counting multiplicity) of the Laurent polynomial $Q_R(x, y_0)$ inside the circle $\TT_{\alpha}^1,$ and $P_{\alpha, y_0, R}^1$ denote the order of pole of $Q_{R}(x, y_0)$ at $x =0.$ Let $\nu_{\alpha, R}^1(y_0) := Z_{\alpha, y_0, R}^1 - P_{\alpha, y_0, R}^1.$ From Proposition \ref{allin1} (when applied to the torus $\TT_{\alpha, \beta}^2$), it follows that $\nu_{\alpha, R}^1 (y)$ is constant for all $y$ in $\TT_{\beta}^1.$ We define $\nu_{\alpha, \beta, R}^1 := \nu_{\alpha, R}^1(y) \in \Z,$ for all $y \in \TT_{\beta}^1.$ Therefore, \eqref{partial1} can be simplified to\begin{equation}\label{alpha}
 \frac{\partial \tilde{\m}_{\alpha, \beta}(R)}{\partial \alpha} = \frac{\nu_{\alpha, \beta, R}^1}{\alpha}.
\end{equation}  Similarly, \begin{equation}\label{beta}
\frac{\partial \tilde{\m}_{\alpha, \beta}(R)}{\partial \beta} = \frac{\nu_{\alpha, \beta, R}^2}{\beta},
\end{equation} where $\nu_{\alpha, \beta, R}^2 = Z_{\alpha, \beta, R}^2 - P_{\alpha, \beta, R}^2.$ Here $Z_{\alpha, \beta, R}^2$ and $P_{\alpha, \beta, R}^2$ are similarly defined.

Since the integer-valued functions $\nu_{\alpha, \beta, R}^1$ and $\nu_{\alpha, \beta, R}^2$ depend on $\alpha$ and $\beta$ continuously, they are constant on $\tilde{W}_{a, b} \subset \ \mbox{int}(\mm_{a, b}).$ In other words, \[\nu_{a, b, R}^1 = \nu_{\alpha, \beta, R}^1, \quad \mbox{and} \quad \nu_{a, b, R}^2 = \nu_{\alpha, \beta, R}^2, \quad \quad \mbox{for all} \ (\alpha, \beta) \in \tilde{W}_{a, b}.\] Integrating \eqref{alpha} with respect to $\alpha$ and then taking the real part yields \[\m_{\alpha, \beta}(R) = \m_{1, 1}(R) + \nu_{a, b, R}^1 \log \alpha + F(\beta),\] where $F$ is a function of $\beta$ which does not depend on $\alpha$ and $R.$ A similar process when applied to \eqref{beta} implies that \[\m_{\alpha, \beta}(R) = \m_{1, 1}(R) + \nu_{a, b, R}^2 \log \beta + G(\alpha),\] where $G$ is independent of $\beta$ and $R.$ From the above equalities and \eqref{new1}, we conclude that \begin{equation}\label{new2}
\m_{\alpha, \beta}(R) = \m_{1, 1}(R) + \nu_{a, b, R}^1 \log \alpha + \nu_{a, b, R}^2 \log \beta + c,
\end{equation} for all $(\alpha, \beta) \in \mm_{a, b},$ and some constant $c$ independent of $\alpha, \beta, R.$ As $|R| > R_{1, 1},$ evaluating \eqref{new2} at $\alpha = 1, \beta = 1$ we obtain $c = 0.$ Then, combining \eqref{1eq} and \eqref{new2} together, we derive that \begin{equation}\label{finalnew}
f(a, b) =  \nu_{a, b, R}^1 \log a + \nu_{a, b, R}^2 \log b , \quad \quad \quad \mbox{for} \ r \in O_{a, b}.
\end{equation} Since $f(a, b)$ in \eqref{1eq} is independent of $r,$ comparing \eqref{finalnew} with \eqref{1eq} we obtain that, for $j =1, 2,$ $\nu_{a, b, R}^j$  is constant in $O_{a, b},$ i.e. \[\nu_{a, b, R}^j = \nu_{a, b, r}^j, \quad \quad \mbox{when} \ r \in O_{a, b}, j \in \{1, 2\}.\] This concludes the proof of Theorem \ref{mainthm}, namely \[\m_{a, b}(r) = \m_{1, 1}(r) + \nu_{a, b, r}^1 \log a + \nu_{a, b, r}^2 \log b , \quad \quad \quad \mbox{for} \ r \in O_{a, b} = U_{a, b} \cap U_{1, 1}.\] 
\end{proof}

\begin{rem}
    Theorem \ref{mainthm} can be explained in terms of in terms of periods when the curve defined by the polynomial has non-zero genus. Following the investigations by Rodriguez-Villegas \cite{RV1}, Deninger \cite{Den1} et al., we conclude that \[\frac{d \tilde{\m} (Q_r)}{d r} =  \frac{1}{(2\pi i)^2}\int_{\TT^2}\frac{1}{r - Q(x, y)}\frac{dx}{x}\frac{dy}{y}\] is a period of the non-singular curve $C_r$ associated to $Q_r$ (when the genus of $C_r$ is non-zero for generic $r$), and in particular, from Proposition $1.2$ in \cite{Den1}, we can argue that $\frac{d \tilde{\m} (Q_r)}{d r}$ remains invariant under the continuous deformation of the integration torus $\TT^2$ as long as the deformation process does not reach any points of $C_r$ (see \cite{Gr1}). Therefore, in our case, $\frac{d \tilde{\m} (Q_r)}{d r} = \frac{d \tilde{\m}_{a, b} (Q_r)}{d r}$ as long as $\TT^2$ is continuously deformed into $\TT_{a, b}^2$ without approaching a point of $C_r.$ Further, the coefficients of $\log a$ and $\log b$ can also be derived combining our method with Proposition $1.2$ in \cite{Den1}.
\end{rem}

Theorem \ref{nvarthm} follows from generalizing the argument in the above proof. Indeed, let $K_{\mathfrak{a}} := \max_{k \in \K_{\mathfrak{a}}} |k|$, and let $K_{\mathfrak{a}, \mathfrak{1}} = \max\{K_{\mathfrak{a}}, K_{\mathfrak{1}}\},$ where $\mathfrak{1} = (1, \dots, 1).$ Then, generalizing the steps in \cite{RV1} and \cite{Ber1}, we define \begin{equation}\label{nvar}
\tilde{\m}_{\mathfrak{a}}(P_k) = \log k - \sum_{m \geq 0}\frac{a_{m, \mathfrak{a}}}{m}k^{-m}, \quad |k| > K_{\mathfrak{a}, \mathfrak{1}}, k \notin (-\infty, 0],
\end{equation}
where $\log$ denotes the principal branch of the logarithm, and $a_{m, \mathfrak{a}}$ is defined as follows: \begin{align*}
a_{m, \mathfrak{a}} =& \left[\frac{1}{(2\pi i)^n}\int_{\TT_{\mathfrak{a}}^2} \frac{dx_1 \cdots dx_n}{x_1 \cdots x_n (1-r^{-1}P(x_1, \dots, x_n))}\right]_m = \frac{1}{(2\pi i)^n}\int_{\TT_{\mathfrak{a}}^n} P(x_1, \dots, x_n)^n\frac{dx_1}{x_1} \cdots \frac{dx_n}{x_n}, 
\end{align*}
where $[T(s)]_m$ denotes the coefficient of $s^{-m}$ in the series $T(s).$ 

It is immediate to see that $\tilde{\m}_{\mathfrak{a}}(P_k)$ is holomorphic in the region defined by the intersection of $\{|k| > K_{\mathfrak{a}, \mathfrak{1}}\}$ and $\C \setminus (-\infty, 0]$. Also, \[\Re(\tilde{\m}_{\mathfrak{a}}(P_k)) = \m_{\mathfrak{a}}(P_k), \quad |k| > K_{\mathfrak{a}, \mathfrak{1}}.\] A similar argument as in the $2$-variable case shows that $a_{m, \mathfrak{a}} = a_{m, \mathfrak{1}}$ for all $m \geq 0,$ and therefore, we have the following equality: \begin{equation}\label{nvareq}
\frac{d\tilde{\m}_{\mathfrak{a}}(P_k)}{d k} = \frac{d\tilde{\m}_{\mathfrak{1}}(P_k)}{d k}, \quad \quad \mbox{for} \ |k| > K_{\mathfrak{a}, \mathfrak{1}}.
\end{equation} A similar argument as in the $2$-variable case also shows that $\frac{d\tilde{\m}_{\mathfrak{a}}(P_k)}{d k}$ is holomorphic in $\{|k| > K_{\mathfrak{a}, \mathfrak{1}}\}.$

Now integrating both sides of \eqref{nvareq} with respect to $k$ and then taking real parts on both sides yield that, for $\{|k| > K_{\mathfrak{a}, \mathfrak{1}}\},$ \begin{equation}\label{nvarequality}
\m_{\mathfrak{a}}(P_k) = \m_{\mathfrak{1}}(P_k) + g(\mathfrak{a}).
\end{equation}  
 
Let $U_{\mathfrak{a}}$ be the unbounded open connected component of $\C \setminus \K_{\mathfrak{a}}$ which contains the region $\left\{|k| > K_{\mathfrak{a}, \mathfrak{1}}\right\}.$ As both sides of \eqref{nvarequality} are harmonic on $O_{\mathfrak{a}} := U_{\mathfrak{a}} \cap U_{\mathfrak{1}},$ the equality can be extended to $O_{\mathfrak{a}}.$ In other words, for $k \in O_{\mathfrak{a}},$ we have $\m_{\mathfrak{a}}(P_k) = \m_{\mathfrak{1}}(P_k) + g(\mathfrak{a}).$ 

It only remains to express $g$ explicitly in terms of $\mathfrak{a}.$ Consider the functions \[a_j\frac{\partial \tilde{\m}_{\mathfrak{a}}(P_k)}{\partial a_j} = \frac{1}{(2\pi i)^n} \int_{|x_1| = a_1, \dots, \widehat{|x_j| = a_j}, \dots, |x_n| = a_n }\left(\int_{|x_j| = a_j}\frac{\partial_{x_j} P_k}{P_k} dx_j \right) \frac{dx_1}{x_1}\cdots \widehat{\frac{dx_j}{x_j}}\cdots \frac{dx_n}{x_n}\] for all $j = 1, \dots, n.$ Here $\ \widehat{•} \ $ denotes that the term is omitted from the expression. Now, again following the steps for $2$-variable case we conclude that $a_j\frac{\partial \tilde{\m}_{\mathfrak{a}}(P_k)}{\partial a_j}$ is constant depending only on $a_j.$ More precisely, we find that \begin{equation}\label{nvarnu}
 a_j\frac{\partial \tilde{\m}_{\mathfrak{a}}(P_k)}{\partial a_j} = \nu_{\mathfrak{a}, k}^j,
 \end{equation}where $\nu_{\mathfrak{a}, k}^j$ is the difference between the number of zeroes (counting multiplicities) of $P_{k}(a_1, \dots, a_{j-1}, x_j, a_{j+1}, \dots, a_n)$ inside the circle $\TT_{a_j}^1,$ denoted by $Z_{\mathfrak{a}, k}^j,$  and the order of the pole of $P_{k}(a_1, \dots, a_{j-1}, x_j, a_{j+1}, \dots, a_n)$ at $x_j = 0,$ denoted by $P_{\mathfrak{a}, k}^j.$ In other words, \[\nu_{\mathfrak{a}, k}^j = Z_{\mathfrak{a}, k}^j - P_{\mathfrak{a}, k}^j.\] We also note that $\nu_{\mathfrak{a}, k}^j$ is independent of $k$ when $k \in O_{\mathfrak{a}},$ and only depends on $\mathfrak{a}$ and the polynomial $P = k - P_k.$ Integrating \eqref{nvarnu} with respect to $a_j$ for $j = 1, \dots, n,$ derives Theorem \ref{nvarthm}.

\section{Proof of Theorem \ref{constant}}

In this section, our goal is to provide the proof of Theorem \ref{constant}, and eventually evaluate $\m_{a, b}(Q_r)$ when $r \in \C \setminus (\Rh_{a, b} \cup U_{a, b}).$ Our proof uses Proposition \ref{allin1} to conclude that, for all $r$ from a small enough neighbourhood in one of the bounded regions under consideration, certain properties of the roots of $Q_r (a, y)$ or $Q_r(x, b)$ remain invariant. This, combined with the properties of harmonic functions along with Rouch\'e's theorem, gives us the desired results.

Recall that $Q_r(x, y),$ considered as a polynomial in $y$ of degree $d_y$ with coefficients in $\overline{\C(x)},$ can be factored in $\overline{\C(x)}[y]$ as \begin{align}
Q_r(x, y) =& (y)^{-v_2}\left(Q_{F, r}^y (x)(y)^{d_y} + Q_{f, r}^y(x) + \sum_{j=1}^{d_y-1}a_{j, r}^y(x)(y)^j\right) \label{expansn1} \\ =& (y)^{-v_{2}}Q_{F, r}^y(x)\prod_{j=1}^{d_y} (y - y_{j, r} (x)), 
\label{1factor2}
\end{align}  
where the $y_{j, r}(x)$ are algebraic functions in $x,$ $v_2$ is the order of the pole of $Q_r(a, y)$ at $y=0,$ and $Q_{F, r}^y(x)$ and $Q_{f, r}^y(x)$ are the respective leading and "constant" coefficient with respect to the variable $y.$ Similarly, we can factor $Q_r,$ considered as a polynomial in $x$ of degree $d_x$ with coefficients in $\overline{\C(y)},$ as  \begin{align*}
Q_r(x, y) =& (x)^{-v_1}\left(Q_{F, r}^x (y)(x)^{d_x} + Q_{f, r}^x(y) + \sum_{j=1}^{d_x-1}a_{j, r}^x(y)(x)^j\right)  \\ =& (x)^{-v_{1}}Q_{F, r}^x(y)\prod_{j=1}^{d_x} (x - x_{j, r} (y)), 
\end{align*}  
where the $x_{j, r}(y)$ are algebraic functions in $y,$ $v_1$ is the order of the pole of $Q_r(x, b)$ at $x=0,$ and $Q_{F, r}^x(y)$ and $Q_{f, r}^x(y)$ are the respective leading and "constant" coefficient with respect to the variable $x.$

Let $Z_{F, r}^u =\{z \in \C: Q_{F, r}^u(z) = 0\},$ $Z_{f, r}^u = \{z \in \C: Q_{f, r}^u(z) = 0\},$ where $u = x \ \mbox{or} \ y,$ and $V_{a, b, r_0}$ denotes the bounded open connected component of $\C \setminus \Rh_{a, b}$ containing $r_0.$

Since the proofs of the statements in \ref{item:firstcond} and \ref{item:secndcond} of Theorem \ref{constant} are similar, here we restrict ourselves in proving the statement \ref{item:firstcond}  

\begin{proof}[\textbf{Proof of Theorem \ref{constant}}]

In \eqref{expansn1}, we see that the polynomial $Q_r(x, y)$ can be expressed in terms of $y_{j, r}(x)$ (algebraic functions in $x$), $v_2, Q_{F, r}^y(x)$ and $Q_{f, r}^y(x).$ For simplicity we denote \[Q_{F, r}(x) := Q_{F, r}^y(x), Q_{f, r}(x) := Q_{f, r}^y(x), \ \mbox{and} \ d:= d_y.\]

Proposition \ref{allin1} and the assumption in \ref{item:firstcond} in the statement of Theorem \ref{constant} yield that $\varrho_{a, b, r_0}^2(x) = d \ \mbox{or} \ 0$ for all $x \in \TT_{a}^1.$ In particular, $\varrho_{a, b, r_0}^2(a) = d \ \mbox{or} \ 0,$ depending on whether all the roots of $Q_r(a, y)$ lie entirely inside or entirely outside the circle $\TT_{b}^1.$

The following three cases can occur when $\varrho_{a, b, r_0}^2(a) = d \ \mbox{or} \ 0.$ \begin{enumerate}[\textbf{Case \arabic*}:,ref=\textbf{Case \arabic*}]
\item \label{item:case1} For all $x \in \TT_{a}^1,$ \[Q_{F, r_0}(x) \cdot Q_{f, r_0}(x) \neq 0.\]

\item \label{item:case2} $Q_{F, r_0}^y$ vanishes on $\TT_{a}^1,$ but $Q_{f, r_0}^y$ does not, i.e. \[Z_{F, r_0}^y \cap \TT_{a}^1 \neq \varnothing, \quad \mbox{and} \quad Z_{f, r_0}^y \cap \TT_{a}^1 = \varnothing,\]

\item \label{item:case3}  $Q_{f, r_0}^y$ vanishes on $\TT_{a}^1,$ but $Q_{F, r_0}^y$ does not, i.e. \[Z_{f, r_0}^y \cap \TT_{a}^1 \neq \varnothing, \quad \mbox{and} \quad Z_{F, r_0}^y \cap \TT_{a}^1 = \varnothing.\]
\end{enumerate}
\ref{item:case1}: Since \[Q_{F, r_0}(x) \cdot Q_{f, r_0}(x) \neq 0 \quad \quad \mbox{for all} \ x \in \TT_{a}^1,\] the discussion preceding the proof of Proposition \ref{allin1} implies that the algebraic functions $y_{j, r_0}(x)$ may have only an algebraic branch point at $x = a \in \TT_{a}^1.$ From Proposition \ref{allin1} we know that $\nu_{a, b, r_0}^2$ is constant in $\TT_{a}^1.$ Therefore, we can in fact assume that $x = a$ is not a branch point of $y_{j, r_0}(x)$ for all $j.$ Indeed, if $x = a$ is branch point, then there exists an $x_0 \in \TT_{a}^1$ close enough to $a$ such that $x_0 \notin S_r,$ where $S_r$ is given in \eqref{criticalpoint}. We replace $a$ with $x_0$ in the statement, and proceed. Here we provide a proof of \ref{item:case1} when $\varrho_{a, b, r_0}^2(a) = d,$ since the case when $\varrho_{a, b, r_0}^2(a) = 0$ is similar. Recall that the condition $\varrho_{a, b, r_0}^2(a) = d$ (resp. $\varrho_{a, b, r_0}^2(a) = 0$) is equivalent to the condition that all the roots of $Q_{r_0}(a, y)$ lie inside (resp. outside) the circle $\TT_{b}^1.$

The polynomial $Q_r(x, y)$ has additional structure: $Q_r(x, y) = r - Q(x, y)$ where $Q$ does not contain any constant term, and $r$ is the constant coefficient in $Q_r.$ Therefore, after multiplying $Q_r$ by $y^{v_2},$ we find from \eqref{expansn1} that one, and only one, among the set of the coefficients \[\mathrm{Coeff}_{Q_r, x}:= \{Q_{F, r}(x), Q_{f, r}(x), a_{1, r}(x), \dots, a_{d-1, r}(x)\} \subset \overline{\C(x)}\] contains $r$ as its constant term, namely the coefficient of $y^{v_2}$ in $y^{v_2}Q_r(x, y).$ Let $b_{v_2, r}(x)$ denotes the said coefficient. Then $b_{v_2, r}(x) \in \mathrm{Coeff}_{Q_r, x},$ and $b_{v_2, r}(x) - b_{v_2, r_0}(x) = r - r_0.$ Since all the coefficients, except $b_{v_2, r},$ do not depend on $r$ by construction, the above discussion further implies that \begin{align}
&\left\{\left|Q_{F, r}(x) - Q_{F, r_0}(x)\right|, \left|Q_{f, r}(x) - Q_{f, r_0}(x)\right|\right\} \cup \left\{\left|a_{j, r}(x) - a_{j, r_0}(x)\right|: 1 \leq  j \leq d-1\right\} \label{F1} \\ =& \left\{0, \left|b_{v_2, r}(x) - b_{v_2, r_0}(x)\right|\right\}  = \left\{0, \left|r -r _0\right|\right\}. \nnum
\end{align}
In other words, if, for example, $Q_{F, r}(x) = b_{v_2, r}(x),$ then \[Q_{F, r}(x) - Q_{F, r_0}(x) = r - r_0, \quad Q_{f, r}(x) = Q_{f, r_0}(x), \quad \mbox{and, for all} \ j, \ a_{j, r}(x) = a_{j, r_0}(x).\]

Next we investigate the relation between $|Q_r(a, y) - Q_{r_0}(a, y)|$ and $|Q_{r_0}(a, y)|$ when $y$ takes values in certain sufficiently small circles. 

Let \[\epsilon_{ij} = \frac{1}{b}\left|y_{i, r_0}(a) - y_{j, r_0}(a)\right| \quad \quad \mbox{and} \quad \quad \epsilon_{k} = \frac{1}{b}\min_{t \in \TT_{b}^1}\left|y_{k, r_0}(a) - t\right|.\] Since all the roots of $Q_{r_0}(a, y)$ are distinct and lie inside the circle $\TT_{b}^1$, the quantities $\epsilon_{ij}, \epsilon_{k}$ are non-zero for any $i, j, k \in \{1, \dots, d\}$ such that $i \neq j.$

We denote \[\Upsilon = \min_{\substack{1 \leq i < j \leq d \\ 1 \leq k \leq d}}\left\{\epsilon_{ij}, \epsilon_{k}\right\}.\] Note that $\Upsilon > 0.$ Let $\epsilon \in \left(0, \Upsilon \right)\cap \left(0, 1\right).$ We define the closed discs \[D_j = \{z: |z - y_{j, r_0}(a)| \leq \epsilon\}, \quad \mbox{for} \ j =1, \dots, d.\] Let $C_j = \partial D_j$ be the boundary of $D_j.$ The choice of $\epsilon$ then confirms that the discs $D_j$ are disjoint and $Q_{r_0}(a, y)$ does not vanish on $C_j.$ This implies $\psi_{j, \epsilon, r_0} := \min_{y \in C_j}|Q_{r_0}(a, y)|$ is positive for each $j.$

Let $\delta_{j, r_0, \epsilon} := \frac{\psi_{j, \epsilon, r_0}}{d+1}.$ Then, for $y \in C_j,$ and $r \in V_{a, b, r_0}$ such that \[|r - r_0| <  \delta_{j, r_0, \epsilon},\] we have \begin{align*}
&\left|Q_r(a, y) - Q_{r_0}(a, y)\right|  \\ =& \left|\left(Q_{F, r}(a) - Q_{F, r_0}(a)\right)(y)^d + \left(Q_{f, r}(a) - Q_{f, r_0}(a)\right) + \sum_{j=1}^{d-1}\left(a_{j, r}(a) - a_{j, r_0}(a)\right)(y)^j \right|  \\ \leq & |r-r_0|\left(\sum_{j = 0}^d |\epsilon|^j\right) \leq  (d+1)|r - r_0|  < \psi_{j, \epsilon, r_0} \leq \left|Q_{r_0}(a, y)\right|,
\end{align*}
where the first inequality follows from \eqref{F1}. 

 This implies that, for $j = 1, \dots, d,$ \[\left|Q_r(a, y) - Q_{r_0}(a, y)\right| < \left|Q_{r_0}(a, y)\right|\] on $C_j.$ Therefore, it follows from Rouch\'e's Theorem that $Q_r(a, y)$ and $Q_{r_0}(a, y)$ have the same number of root(s) in the interior of $D_j$ when $|r-r_0| <  \delta_{j, r_0, \epsilon}.$ Moreover, for \[\delta(\epsilon, r_0) = \min_{1 \leq j \leq d} \delta_{j, r_0, \epsilon} > 0,\] the choice of $\epsilon$ implies that, when $|r - r_0| < \delta(\epsilon, r_0),$ all the roots of $Q_{r}(a, y)$ lie entirely inside the circle $\TT_{b}^1.$ 

When $|r - r_0| < \delta(\epsilon, r_0),$ another application of Proposition \ref{allin1} yields that all the roots of $Q_r(x, y)$ lie inside $\TT_{b}^1$ for every $x \in \TT_{a}^1.$ Following the discussion in Section \ref{ATaMM} regarding the Mahler measure over arbitrary tori, we conclude that, for $r \in \{z: |z-r_0| < \delta_{\epsilon, r_0}\} \subset V_{a, b, r_0},$ \begin{align}
\m_{a, b}(Q_r(x, y)) &= \m_{a, b} \left((y)^{-v_{2}}Q_{F, r}(x)\prod_{j=1}^d (y - y_{j, r} (x))\right) \nnum \\ &= \m_{a}(Q_{F, r}(x)) - v_2 \log b + \m_{a, b}\left(\prod_{j=1}^d (y - y_{j, r} (x))\right) \nnum \\ &= \m_a(Q_{F, r}(x)) -v_2 \log b + d\log b. \label{eval}
\end{align}

Similarly, when all roots of $Q_{r_0}(a, y)$ lie outside the circle $\TT_{b}^1,$ we have for $r \in \{z: |z-r_0| < \delta_{\epsilon, r_0}\} \subset V_{a, b, r_0},$ \begin{align}
\m_{a, b}(Q_r(x, y)) =&  \m_{a}(Q_{F, r}(x)) - v_2 \log b + \m_{a}(Q_{f, r}(x)) - \m_{a}(Q_{F, r}(x)) \nnum \\ =& \m_{a}(Q_{f, r}(x)) - v_2 \log b. \label{eval*}
\end{align}

Recall that, $\nu_{a, b, r}^2$ denotes the difference between the number of zeros (counting multiplicity) of $Q_r(a, y)$ inside $\TT_{b}^1$ and the order of pole of $Q_{r}(a, y)$ at $y =0.$ Then the above discussion implies that, for $r \in \{z: |z-r_0| < \delta_{\epsilon, r_0}\} \subset V_{a, b},$ \[\nu_{a, b, r}^2 = \nu_{a, b, r_0}^2 = \varrho_{a, b, r_0}^2(x) - v_2 = d - v_2 \ \mbox{or} \ - v_2.\]

Since $\m_{a}(Q_{F, r}(x))$ and $\m_{a}(Q_{f, r}(x))$ are harmonic, and $\m_{a, b}(Q_r(x, y))$ is harmonic for all $r \in V_{a, b, r_0} \setminus \mathcal{S}_{a, b, r_0}$ (where $\mathcal{S}_{a, b, r_0}$ is a finite set containing all the $r \in V_{a, b, r_0}$ such that $Q_r(x, y)$ is singular), the equalities in \eqref{eval} and \eqref{eval*} can be extended to a larger set $ V_{a, b, r_0} \setminus \mathcal{S}_{a, b, r_0}.$ using the harmonicity of Mahler measure. In other words, for $r, r_0 \in V_{a, b, r_0} \setminus \mathcal{S}_{a, b, r_0},$ \begin{equation}\label{evalconst}
\m_{a, b}(Q_r) - \nu_{a, b, r_0}^2 \log b = \left\{\begin{array}{ll}
\m_{a}(Q_{F, r}(x)) & \mbox{all roots of} \ Q_{r_0}(a, y) \ \mbox{lie inside} \ \TT_{b}^1, \\ \m_{a}(Q_{f, r}(x)) & \mbox{all roots of} \ Q_{r_0}(a, y) \ \mbox{lie outside} \ \TT_{b}^1.
\end{array} \right.
\end{equation}
By continuity, \eqref{evalconst} holds for all $r \in V_{a, b, r_0},$ and this concludes the proof of the \ref{item:case1}.

\smallskip

Recall that  $Z_{F, r}^y =\{z \in \C: Q_{F, r}^y(z) = 0\},$ and $Z_{f, r}^y = \{z \in \C: Q_{f, r}^y(z) = 0\}.$ \newline
\ref{item:case2}: If \[Z_{F, r_0}^y \cap \TT_{a}^1 \neq \varnothing, \quad \mbox{and} \quad Z_{f, r_0}^y \cap \TT_{a}^1 = \varnothing,\] then there exists $x' \in Z_{F, r_0}^y \cap \TT_{a}^1,$ and a $l \in \{1, \dots, d_y\}$ such that $y_{l, r_0}(x)$ has a pole at $x'.$ Then Proposition \ref{allin1} and the conditions in the statement of Theorem \ref{constant} imply that all the roots of $Q_{r_0}(a, y)$ lie outside the circle $\TT_{b}^1,$ and we can choose an $x_0 \in \TT_{a}^1$ in a sufficiently small neighbourhood of $x',$ such that $x_0$ is not a pole of $y_{j, r_0}$ for all $j.$ Such choice is possible since the set of critical points $S_{r_0}
$ of the global analytic function $\textbf{y}_{r_0}$ is a finite set. Then a similar argument as in \textbf{Case $1$} implies that, for all $r \in V_{a, b, r_0},$ \[\m_{a, b}(Q_r) - \nu_{a, b, r}^2 \log b = \m_{a}(Q_{f, r}^y(x)).\]
\ref{item:case3}: If \[Z_{f, r_0}^y \cap \TT_{a}^1 \neq \varnothing, \quad \mbox{and} \quad Z_{F, r_0}^y \cap \TT_{a}^1 = \varnothing,\] then there exists $x'' \in Z_{f, r_0}^y \cap \TT_{a}^1,$ and a $p \in \{1, \dots, d_y\}$ such that $y_{p, r_0}(x)$ has a zero at $x''.$ Again, Proposition \ref{allin1} and the conditions in the statement of Theorem \ref{constant} imply that all the roots of $Q_{r_0}(a, y)$ lie inside the circle $\TT_{b}^1,$ and we can choose an $x_1 \in \TT_{a}^1,$ such that $x_1 \notin S_{r_0} \cup Z_{f, r_0}^y,$ and $Q_{r_0}(x_1, y)$ has all the roots inside $\TT_{b}^1.$ With these conditions, we have, for all $r \in V_{a, b, r_0},$ \[\m_{a, b}(Q_r) - \nu_{a, b, r}^2 \log b = \m_{a}(Q_{F, r}^y(x)).\]

This concludes the proof of the statement \ref{item:firstcond}. Statement \ref{item:secndcond} follows from an analogous argument.
\end{proof} 

Next, we sketch a proof of Theorem \ref{nvarconstant}. Recall that multiplying $P_k$ with a suitable power of $x_{j},$ we can factorise $P_k$ in linear factors with coefficients in $\overline{\C(x_1, \dots, \widehat{x_j}, \dots, x_{n})}$ as \[P_{k}(x_1, \dots, x_n) = x_j^{-v_j}P_{F, k}^j(x_1, \dots, \widehat{x_j}, \dots, x_{n}) \prod_{l = 1}^{d_n} \left(x_j - X_{l, k, j}\left(x_1, \dots, \widehat{x_j}, \dots, x_{n}\right)\right),\] where $d_j$ is the degree of $P_k$ as a polynomial in $x_j,$ $X_{l, k, j}$ are algebraic functions of $(x_1, \dots, \widehat{x_j}, \dots, x_{n})$ for $l = 1, \dots, d_n,$ $P_{F, k}^j$ is the leading coefficient with respect to the variable $x_j,$ and $v_j$ is the largest power of $x_j^{-1}$ in $P_k.$ Let $P_{f, k}^j(x_1, \dots, \widehat{x_j}, \dots, x_{n})$ denote the "constant" coefficient with respect to the variable $x_j.$ Then \[P_{F, k}^j(x_1, \dots, \widehat{x_j}, \dots, x_{n}) \prod_{j = 1}^{d_n} X_{l, k, j}(x_1, \dots, \widehat{x_j}, \dots, x_{n}) = P_{f, k}^j(x_1, \dots, \widehat{x_j}, \dots, x_{n}).\] 

For $\left(u_1, \dots, \widehat{u_j}, \dots, u_n\right) \in \TT_{a_1, \dots, \widehat{a_j}, \dots, a_n}^{n-1},$ let $\varrho_{\mathfrak{a}, k}^j\left(u_1, \dots, \widehat{u_j}, \dots, u_n\right)$ be the number of zeroes (counting multiplicities) of $P_{k}(u_1, \dots, u_{j-1}, x_j, u_{j+1}, \dots, u_n)$ inside the circle $\TT_{a_j}^1.$ Then, from the above discussion, we have $P_{\mathfrak{a}, k}^j = v_j,$ and \[\varrho_{\mathfrak{a}, k}^j\left(a_1, \dots, \widehat{a_j}, \dots, a_n\right) = Z_{\mathfrak{a}, k}^j = \nu_{\mathfrak{a}, k}^j + P_{\mathfrak{a}, k}^j = \nu_{\mathfrak{a}, k}^j + v_j.\] An analogous argument as in the proof of Proposition \ref{allin1} yields the following proposition. 

\begin{prop}\label{allin1n}
Let $k \notin \K_{\mathfrak{a}}.$ Then $\varrho_{\mathfrak{a}, k}^j (x_1, \dots, \widehat{x_j}, \dots, x_{n})$ is constant for all $(x_1, \dots, \widehat{x_j}, \dots, x_{n}) \in \TT_{a_1, \dots, \widehat{a_j}, \dots, a_n}^{n-1}.$
\end{prop} 

We omit the proof of the proposition here since it is an immediate extension of Proposition \ref{allin1}, which follows from an induction argument on $n \geq 2.$ 

Then Proposition \ref{allin1n}, along with a similar argument as in the proof of Theorem \ref{constant}, establishes Theorem \ref{nvarconstant}.

\section{Generalized Mahler measure of a family of polynomials}
In this section, we consider the family of polynomials \[\left\{Q_r(x, y) = x + \frac{1}{x} + y + \frac{1}{y} + r : r \in \C \right\}.\] Boyd \cite{Bo2}, Deninger \cite{Den1}, Rodriguez-Villegas \cite{RV1}, Lal\'in \cite{LR07}, Rogers and Zudilin \cite{RZ1} et al have successfully evaluated the (standard) Mahler measure of $Q_r$ for different values of $r$ in terms of special values of Dirichlet $L$-functions, $L$-functions of elliptic curves, special values of the Bloch-Wigner dilogarithm, etc. Our aim here is to apply Theorems \ref{mainthm} and \ref{constant} to evaluate the generalized Mahler measure of $Q_r.$ 

Before proceeding with this evaluation we recall some notation associated to the considered family of polynomials for the reader's convenience. \begin{enumerate}
\item[1.] The map in \eqref{Rh} is defined in this case as \[q: \TT_{a, b}^2 \mapsto \C, \quad \quad (x, y) \mapsto  x + \frac{1}{x} + y + \frac{1}{y}.\]

\item[2.] The image of $q$ is denoted by $\Rh_{a, b}.$ The elements of $\Rh_{a, b}$ are of the form \begin{equation*}
 r = \left(a + a^{-1}\right)\cos \alpha + \left(b + b^{-1}\right)\cos \beta + i\left[\left(a - a^{-1}\right)\sin \alpha + \left(b - b^{-1}\right)\sin \beta\right],
\end{equation*} where $\alpha, \beta \in [-\pi, \pi).$

\item[3.] Since $\Rh_{a, b}$ is compact, $R_{a, b} = \max_{r \in \Rh_{a, b}}|r|$ exists. 

\item[4.] $U_{a, b}$ denotes the unbounded open connected component of $\C \setminus \Rh_{a, b}.$ It contains the region $\{|r| > R_{a, b}\};$ since $R_{1, 1} = 4,$ we have $U_{a, b} \subseteq U_{1, 1}.$
\end{enumerate}

Now we are ready to apply our theorems to evaluate the generalized Mahler measure of $Q_r.$

\subsection{Generalized Mahler measure on the unbounded component of $\C \setminus \Rh_{a, b}$}

In \cite{RV1} Rodriguez-Villegas expressed the (standard) Mahler measure of $Q_r$ in terms of Eisenstein-Kronecker series for any $r \in \C.$ Combining his proof and  Theorem \ref{mainthm}, we will show that, for fixed $a, b >0,$ there exists a large open subset of $\C,$ namely $O_{a, b} = U_{a, b} \cap U_{1, 1},$ such that if $r \in O_{a, b},$ then the Mahler measure remains unchanged irrespective of the dependence of the integration torus on $(a, b).$ We will in fact go further and show that $O_{a, b}$ is the unbounded component of $\C \setminus \Rh_{a, b},$ namely $U_{a, b}.$ Later in this section, we will give an explicit expression of the region $O_{a, b},$ as well as of the region $\Rh_{a, b}.$

 Recall that, for fixed $a, b > 0,$ $Q_r$ does not vanish on $\TT^2_{a, b}$ if and only if $r \notin \Rh_{a, b}.$ In order to show that, for a fixed $a, b >0,$ \[\m_{a, b}(Q_r) = \m (Q_r), \quad \quad \quad \mbox{for all} \ r \in O_{a, b},\] it suffices to evaluate $\nu_{a, b, r}^j$ for $j =1, 2.$ Since these quantities are constant in the region $O_{a, b},$ we can choose a suitable $r$ and apply Theorem \ref{mainthm} to evaluate them. Let \[R = R_{a, b} + R_{1, 1} = a + \frac{1}{a} + b + \frac{1}{b} + 4.\] Note that $R \in O_{a, b}$ and $R \notin (-\infty, 0].$ 

Recall that $\nu_{a, b, r}^1$ denotes the difference between the number of zeros (counting multiplicity), namely $Z_{a, b, r}^1$ and the number of poles (counting multiplicity), namely $P_{a, b, r}^1,$ of $Q_r(x, b)$ inside the circle $\TT_{a}^1,$ i.e. \[\nu_{a, b, r}^1 = Z_{a, b, r}^1 - P_{a, b, r}^1,\] and that $\nu_{a, b, r}^2$ is also defined in a similar way. 

Since $Q_R(x, b)$ is holomorphic everywhere except for a simple pole at $x = 0,$ we have $P_{a, b, R}^1 = 1.$ Therefore, $xQ_R(x, b)$ has no pole in $\C.$ Now $xQ_R(x, b)$ can be factored in $\C[x]$ as \[xQ_R(x, b) = \left(x - x_+\right)\left(x - x_-\right),\] where \[x_{\pm} = \frac{-\left(R + b + \frac{1}{b}\right)\pm \sqrt{\left(R + b + \frac{1}{b}\right)^2 - 4}}{2}.\] Notice that $x_+ \cdot x_- = 1,$ and since $R + b + \frac{1}{b} > a + \frac{1}{a},$ we also have \begin{align*}
|x_-| &= \left|\frac{R + b + \frac{1}{b}+ \sqrt{\left(R + b + \frac{1}{b}\right)^2 - 4}}{2}\right| = \frac{R + b + \frac{1}{b}+ \sqrt{\left(R + b + \frac{1}{b}\right)^2 - 4}}{2} \\ &= \frac{R + b + \frac{1}{b}+ \sqrt{\left( a + \frac{1}{a} + b + \frac{1}{b} + b + \frac{1}{b} + 6\right)\left( a + \frac{1}{a} + b + \frac{1}{b}  + b + \frac{1}{b} + 2\right)}}{2} \\ &\geq a + \frac{1}{a}. 
\end{align*}    
 Since $a + \frac{1}{a} > \max\left\{a, \frac{1}{a}\right\},$ we have $|x_+| \leq \frac{1}{a + \frac{1}{a}} < a,$ and therefore, $Z_{a, b, R}^1 = 1.$ By the definition of $\nu_{a, b, R}^1,$ it follows that $\nu_{a, b, R}^1 = 0.$ A similar argument shows that $\nu_{a, b, R}^2 = 0.$ Combining Theorem \ref{mainthm} and the values obtained above, we derive that, for $r \in U_{a, b} \subset U_{1, 1},$ \[\m_{a, b}(Q_r) = \m(Q_r),\] and the required $O_{a, b}$ is in fact the region $U_{a, b}.$

Until now we have been fixing $a, b > 0$ in our discussion. Next, we want to show that our theorem can even be applied to a fixed suitable $r$ in order to obtain certain values of $(a, b)$ such that the equality $\m_{a, b}(Q_r) = \m(Q_r)$ still holds. 

For some particular values of $r \in \R \cup i\R,$ the standard Mahler measure of $Q_r$ has been proven to be the same as (up to a rational multiple) a special value of $L$-function of the elliptic curve corresponding to $Q_r$ due to Boyd \cite{Bo2}, Rodriguez-Villegas \cite{RV1}, Deninger \cite{Den1}, Rogers and Zudilin \cite{RZ1}, Lal\'in and Rogers \cite{LR07} et al. Therefore, an interesting direction would be to search for values of $(a, b)$ such that changing the integration torus from $\TT^2 $ ($=\TT_{1, 1}^2$) to $\TT_{a, b}^2$ keeps the Mahler measure fixed. In order to do so, first notice that, for all $r > R_{a, b},$ Theorem \ref{mainthm} implies that \[\m_{a, b}(Q_r) = \m(Q_r).\] Since $a$ and $b$ are fixed arbitrarily, we can fix $r = r_0 > 4,$ and conclude that, for all $2$-tuples $(a, b)$ satisfying \[a + \frac{1}{a} + b + \frac{1}{b} < r_0,\] we have $\m_{a, b}(Q_{r_0}) = \m(Q_{r_0}).$ Since the change of variables $r \mapsto -r$ covers the case when $r < -4,$ it is sufficient to consider the $r > 4$ case here.

For $r \in i\R,$ it suffices to investigate the imaginary part of $r \in \Rh_{a, b}.$ Indeed, once we calculate the $\max_{r \in \Rh_{a, b}} \im(r),$ we can conclude that all $r' \in \C,$ such that $\im(r') >  \max_{r \in \Rh_{a, b}} \im(r),$ belong to the unbounded component of $\C \setminus \Rh_{a, b},$ namely $U_{a, b}.$ The following discussion results in gathering the required $2$-tuples $(a, b)$ such that \[\m_{a, b}(Q_{r_0}) = \m(Q_{r_0})\] for a fixed $r' = r_0.$

Recall that, any element in $\Rh_{a, b}$ can be written as \[r = \left(a + a^{-1}\right)\cos \alpha + \left(b + b^{-1}\right)\cos \beta + i\left[\left(a - a^{-1}\right)\sin \alpha + \left(b - b^{-1}\right)\sin \beta\right],\] where $\alpha, \beta \in [-\pi, \pi).$ Notice that, \[\left|\im(r)\right| = \left|\left(a - a^{-1}\right)\sin \alpha + \left(b - b^{-1}\right)\sin \beta\right| \leq \left|a - a^{-1}\right| + \left|b - b^{-1}\right|,\] and, for $\alpha = \beta \in \{-\frac{\pi}{2}, \frac{\pi}{2}\},$ we have \[r_{\max, i\R} =  i\left[\left|a - a^{-1}\right| + \left|b - b^{-1}\right|\right].\] Therefore, when $a$ and $b$ are fixed, we have $\m_{a, b}(Q_{r}) = \m(Q_{r})$  for all $r \in \left\{z \in i\R: |z| >  \left|a - a^{-1}\right| + \left|b - b^{-1}\right|\right\}.$ Then a similar argument as in the real case shows that, for a fixed $r_0 \in i\R_{> 0},$ the Mahler measure of $Q_{r_0}$ over the integration torus $\TT_{a, b}^2$ is same as the standard Mahler measure, i.e. \[\m_{a, b}(Q_{r_0}) = \m(Q_{r_0}),\] for all the $2$-tuples $(a, b)$ satisfying \[\left|a - a^{-1}\right| + \left|b - b^{-1}\right| < |r_0|.\] 

Here we mention two such examples for $r = 8$ and $r = 2i.$
\begin{ex}[\textbf{r = 8}] 
 We provide two cases: $(I)$ when $b = a,$ and $(II)$ when $b = \sqrt{a}.$ Notice that, case $(I)$ keeps the symmetry of the polynomial \[Q_8(x, y) = x + \frac{1}{x} + y + \frac{1}{y} + 8\] in the variables $x$ and $y.$ In other words, under the change of variables $x \mapsto y$ and $y \mapsto x,$ the polynomial $Q_8$ remains unchanged, and so does the integration torus $\TT_{a, a}^2$. On the other hand, case $(II)$ breaks the symmetry as then the above changes of variables change the integration torus from $\TT_{a, \sqrt{a}}^2$ to $\TT_{\sqrt{a}, a}^2.$ In spite of the differences between these two cases, there are certain values of $a$ such that \[\m_{a, \sqrt{a}}(Q_8) = \m_{a, a}(Q_8) = \m(Q_8) = 4L'(E_{24}, 0),\] where $E_{24}$ is an elliptic curve of conductor $24$ associated to $Q_8.$ Here the last equality follows from combining the results due to Rogers and Zudilin \cite{RZ1}, and Lal\'in and Rogers \cite{LR07}, where they showed  \[\m(Q_8(x, y)) =\m (Q_2(x, y)) = L'(E_{24}, 0).\]

From the above discussion, we find that, when $(I)$ $a = b,$ the equality \[\m_{a, a}(Q_8) = \m(Q_8)\] holds for all $a$ satisfying \[a + \frac{1}{a} < 4 \Longleftrightarrow 2 - \sqrt{3} < a < 2 + \sqrt{3}.\] Similarly, when $(II)$ $b = \sqrt{a},$ we find that, for \begin{align*}
 & a + \frac{1}{a} + \sqrt{a} + \frac{1}{\sqrt{a}} < 8 \\ \Longleftrightarrow& \frac{17 - \sqrt{41} - \sqrt{2\left(157 - 17\sqrt{41}\right)}}{4} < a < \frac{17 - \sqrt{41} + \sqrt{2\left(157 - 17\sqrt{41}\right)}}{4},
\end{align*} the equality \[\m_{a, \sqrt{a}}(Q_8) = \m(Q_8)\] holds. Since, \[\frac{17 - \sqrt{41} + \sqrt{2\left(157 - 17\sqrt{41}\right)}}{4} > 2 + \sqrt{3}\] and \[\frac{17 - \sqrt{41} - \sqrt{2\left(157 - 17\sqrt{41}\right)}}{4} = \left[\frac{17 - \sqrt{41} + \sqrt{2\left(157 - 17\sqrt{41}\right)}}{4} \right]^{-1},\] we obtain \[\m_{a, \sqrt{a}}(Q_8) = \m_{a, a}(Q_8) = \m(Q_8) = 4L'(E_{24}, 0) \quad \quad \mbox{for all} \ a \in \left(2 - \sqrt{3}, 2 +\sqrt{3}\right).\]

\end{ex}

\medskip

\begin{ex}[\textbf{r = 2$i$}] In 2011, Mellit \cite{W14} showed that \[\m(Q_{2i}) = L'(E_{40}, 0),\] where $E_{40}$ is an elliptic curve of conductor $40,$ associated to $Q_{2i}.$ 

When $b = a,$ Theorem \ref{mainthm} implies that $\m_{a, a}(Q_{2i}) = \m(Q_{2i})$ is true for \[\left|a -a^{-1}\right| < 1 \Longleftrightarrow \frac{\sqrt{5} - 1}{2} < a < \frac{\sqrt{5} + 1}{2}.\] 

Similarly, when $b = \sqrt{a},$ the equality \[\m_{a, \sqrt{a}}(Q_{2i}) = \m(Q_{2i})\] holds for all $a$ satisfying \[\left|a - \frac{1}{a}\right| + \left|\sqrt{a} - \frac{1}{\sqrt{a}}\right| < 2 \Leftrightarrow a_0 < a < a_1,\] where $a_0 \approx 0.530365\dots$ and $a_1 \approx 1.88549\dots$ satisfy $X - 1/X + \sqrt{X} - 1/\sqrt{X} - 2 =0$ and $X - 1/X + \sqrt{X} - 1/\sqrt{X} + 2 =0,$ respectively. Since $a_0 < \frac{\sqrt{5} - 1}{2}$ and $a_1 > \frac{\sqrt{5} + 1}{2},$ we obtain \[\m_{a, a}(Q_{2i}) = \m_{a, \sqrt{a}}(Q_{2i}) = \m(Q_{2i}) = L'(E_{40}, 0)\] for all $a \in \left(\frac{\sqrt{5} - 1}{2}, \frac{\sqrt{5} + 1}{2}\right).$ 

\end{ex}

\subsection{Generalized Mahler measure on bounded component(s) of $\C \setminus \Rh_{a, b}$}

In this section, our goal is to evaluate $\m_{a, b}(Q_r)$ when $ r$ belongs to the bounded connected component(s) of $\C \setminus \Rh_{a, b}.$ In particular, we show there can be at most one such component of $\C \setminus \Rh_{a, b}.$ Later, we apply Theorem \ref{constant} to calculate $\m_{a, b}(Q_r)$ for all $r$ in said component. 

\subsubsection{\textbf{Existence of at most one bounded open connected component of $\C \setminus \Rh_{a, b}$}}
Our aim here is to show that there exists at most one bounded open connected component of $\C \setminus \Rh_{a, b}$ for any $a, b > 0.$

Recall that the elements of $\Rh_{a, b}$ are of the form \begin{equation}\label{expressr}
 r = \left(a + a^{-1}\right)\cos \alpha + \left(b + b^{-1}\right)\cos \beta + i\left[\left(a - a^{-1}\right)\sin \alpha + \left(b - b^{-1}\right)\sin \beta\right],
\end{equation} where $\alpha, \beta \in [-\pi, \pi).$ We have   
 \[R_{a, b} = \max_{r \in \Rh_{a, b}} |r| = a + a^{-1} + b + b^{-1}, \ \mbox{and} \ r_{a, b} := \min_{r \in \Rh_{a, b}} |r| = \left(a + a^{-1}\right) - \left(b + b^{-1}\right).\] Then \[a = b \Longleftrightarrow r_{a, b} = 0 \Longleftrightarrow 0 \in \Rh_{a, b}.\] 

From this point onwards we assume that $a \geq b \geq 1.$ The other cases follow analogously using Lemma \ref{Lem1}. 
\begin{prop}\label{prop0Rh}
There exists at most one bounded open connected component of $\C \setminus \Rh_{a, b},$ and if it exists, then it contains $0.$
\end{prop} 

Note that, if $0 \in \Rh_{a, b},$ then Proposition \ref{prop0Rh} implies that there is no bounded open connected component in $\C \setminus \Rh_{a, b}.$ Before we proceed to prove the proposition, we note some useful properties of $\Rh_{a, b}.$ 
\begin{enumerate}[(\textbf{\textsc{\Alph*}}),ref=property (\textbf{\textsc{\Alph*}})]
\item \label{item:propA} The points $r \in \Rh_{a, b}$ can be interpreted as points on the ellipses \begin{equation}\label{Ell}
E_{b, z} : \left|r - (z + 2)\right| + \left|r - (z - 2)\right| = 2\left(b + b^{-1}\right), 
\end{equation}
where $z \in \C$ lies on the ellipse \begin{equation}\label{ell}
 e_{a} : \left|z - 2\right| + \left|z + 2\right| = 2\left(a + a^{-1}\right).
\end{equation} In other words, elements of $\Rh_{a, b}$ can be identified with points on ellipses $E_{b, z}$ defined by \eqref{Ell} with centres on the ellipse $e_{a}$ in \eqref{ell}. Note that, the centre ($c$) and the foci ($c_1$ and $c_2$) of the ellipse $e_{a}$ are the points $c = 0, c_1 = -2$ and $c_2 = 2.$ The centre ($C_z$) and the foci ($C_{1, z}$ and $C_{2, z}$) of the ellipse $E_{b, z}$ are \[C_{z} = z, C_{1, z} = z-2, \ \mbox{and} \ C_{2, z} = z + 2.\] Any point $p \in \C$ lying inside (resp. outside) the ellipse $E_{b, z}$ satisfies $\left|p - C_{1, z}\right| + \left|p - C_{2, z}\right| < 2\left(b + b^{-1}\right)$ (resp. $\left|p - C_{1, z}\right| + \left|p - C_{2, z}\right| > 2\left(b + b^{-1}\right)$).  Since the length of the minor axis of $e_{a}$ is $2\left(a - a^{-1}\right),$ we derive that, for $z \in e_{a},$ \[\left|\im\left(C_{z}\right)\right| \leq \left(a - a^{-1}\right),\] and, for all $h \in \left(-a + a^{-1}, a - a^{-1}\right),$ there exists a $\tilde{z} \in e_{a}$ such that \begin{equation}\label{imaginaryzh}
\im\left(C_{\tilde{z}}\right) = h.
\end{equation} 

\smallskip

\item \label{item:propB} The region $\Rh_{a, b}$ is symmetric with respect to the imaginary and real axes, i.e. if $r \in \Rh_{a, b},$ then $-\bar{r}, \bar{r} \in \Rh_{a, b}.$ Let $Q_{\pm, \pm}$ denote the four quadrants of $\C,$ namely $Q_{+, +} = \{s \in C: \re(s) \geq 0, \im(s) \geq 0\}, Q_{+, -} = \{s \in C: \re(s) \geq 0, \im(s) \leq 0\}$ and so on. Then, the changes of variable $\{\alpha \mapsto \pi - \alpha, \beta \mapsto \pi - \beta\}$ and $\{\alpha \mapsto -\alpha, \beta \mapsto -\beta\}$ applied to \eqref{expressr} take points $r \in \Rh_{a, b} \cap Q_{+, \pm}$ to $-\bar{r} \in \Rh_{a, b} \cap Q_{-, \pm},$ and $r \in \Rh_{a, b} \cap Q_{\pm, +}$ to $\bar{r} \in \Rh_{a, b} \cap Q_{\pm, -}.$ 
\end{enumerate}

For the rest of this section, $\Rh_{a, b}^c$ denotes the region defined by $\C \setminus \Rh_{a, b}.$

\begin{proof}[\textbf{Proof of Proposition \ref{prop0Rh}}] Recall that $U_{a, b}$ denotes the connected unbounded component of $\Rh_{a, b}^c.$ First we consider the case $a > b.$ 

From the discussion at the beginning of this section, we have $0 \in \Rh_{a, b}^c,$ and moreover the open disc $\{u \in \C: |u| < r_{a, b}\}$ is contained in one of the bounded open connected components of $\Rh_{a, b}^c.$ Let $V_{a, b}$ denote this component. Note that $0 \in V_{a, b}.$

In order to prove the statement, we note that the \ref{item:propB} above implies that we can restrict ourselves to the quadrant $Q_{+, +}.$ 

Since $a > b,$ the ellipse $E_{b, z_0}$ lies completely in the interior of $Q_{+, -},$ where $z_0 = -i\left(a - a^{-1}\right).$  This implies that, for $P \in Q_{+, +},$ we have $\left|P - C_{1, z_0}\right| + \left|P - C_{2, z_0}\right| > 2\left(b + b^{-1}\right),$ where $C_{1, z_0}$ and $C_{2, z_0}$ are the foci of $E_{b, z_0}.$ In other words, the point $P$ lies completely outside the ellipse $E_{b, z_0}$ with centre at $z_0 = -i\left(a - a^{-1}\right) \in e_{a}.$ Let $z_1$ and $z_2$ denote the points $a + a^{-1}$ and $i\left(a - a^{-1}\right),$ respectively. We now consider the following cases: \begin{enumerate}[(\textsc{\textsc{\Roman*}}),ref= (\textsc{\textsc{\Roman*}})]
\item \label{item:propI} $\im(P) >  \im(z_2) = \left(a - a^{-1}\right),$

\item \label{item:propII} $0 \leq \im(P) \leq \im(z_2) = \left(a - a^{-1}\right).$ 
\end{enumerate}

Let $E^{a, b}$ and $e^{a, b}$ denote the boundaries of $U_{a, b}$ and $V_{a, b}$ respectively, i.e. \[E^{a, b} = \overline{U}_{a, b} \setminus U_{a, b}, \quad \mbox{and} \quad e^{a, b} = \overline{V}_{a, b} \setminus V_{a, b}.\] Therefore, $E^{a, b} \cup e^{a, b} \subset \Rh_{a, b}.$ If $P \notin U_{a, b} \cup V_{a, b},$ then $P$ lies in the region bounded by $E^{a, b}$ and $e^{a, b}.$ Then the line $L_{\re(P)}: t = \re(P)$ intersects $E^{a, b}$ at some point $P' \in Q_{+, +}.$ By construction, $\re(P) = \re(P'),$ and $\im(P) \leq \im(P'),$ where the equality holds iff $P = P'.$ Note that $P =P'$ is the trivial case. Therefore, we assume that $P \neq P'.$

\begin{figure}[h]
\centering
\includegraphics[width=110mm,scale=.75]{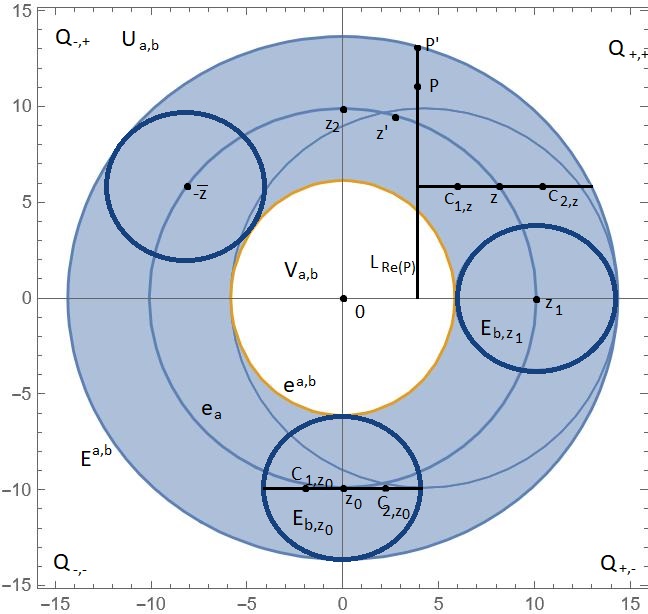}  
\caption{$\Rh_{a, b}, U_{a, b}, V_{a, b}, e_{a}, C_{j, z}, P, P', L_{\re(P)}$ when $a > b$ and \ref{item:propI} holds.}\label{region4}
\end{figure}

\textbf{Claim}: \textit{For $P \in Q_{+, +},$ if \ref{item:propI} holds, and $P \notin V_{a, b} \cup U_{a, b},$ then there exists an ellipse $E_{b, z'},$ with centre at $z' \in e_{a},$ such that \[\left|P - C_{1, z'}\right| + \left|P - C_{2, z'}\right| < 2\left(b + b^{-1}\right).\]}

Since $\im(C_{j, z}) \leq a - a^{-1}$ for all $z \in e_{a}$ and $j = 1, 2,$ case \ref{item:propI} implies that $|P' - C_{j, z}|^2 - |P - C_{j, z}|^2 > |P' - P|^2 > 0$ (see Figure \ref{region4}), and then we have \[\left|P - C_{1, z}\right| + \left|P - C_{2, z}\right| < \left|P' - C_{1, z}\right| + \left|P' - C_{2, z}\right|.\] On the other hand, $P' \in E_{a, b} \subset \Rh_{a, b},$ i.e. there exists a $z' \in e_{a}$ such that $\left|P' - C_{1, z'}\right| + \left|P' - C_{2, z'}\right| = 2\left(b + b^{-1}\right).$ This concludes the proof of the claim.  

Now note that $z \in e_{a}$ can also be written as $z = \left(a + a^{-1}\right)\cos \alpha_z + i \left(a - a^{-1}\right) \sin \alpha_z,$ for $\alpha_z \in [-\pi, \pi).$ Recall that $C_{1, z} = z -2$ and $C_{2, z} = z +2.$ Therefore, when $a$ is fixed, $\left|P - C_{j, z}\right|$ is a continuous function of $\alpha_z$ for $j = 1, 2.$ Let \[\Theta(\alpha_z) := \left|P - C_{1, z}\right| + \left|P - C_{2, z}\right| - 2\left(b + b^{-1}\right)\] define a function from $[-\pi, \pi)$ to $\R.$ From the claim, we have already concluded that, for the case \ref{item:propI}, either $ P \in U_{a, b} \cup V_{a, b},$ or there are $z_0, z' \in e_{b}$ such that \[\left|P - C_{1, z_0}\right| + \left|P - C_{2, z_0}\right| > 2\left(b + b^{-1}\right), \quad \mbox{and} \ \left|P - C_{1, z'}\right| + \left|P - C_{2, z'}\right| < 2\left(b + b^{-1}\right).\]This implies that $\Theta$ is a continuous function which takes both negative and positive values, and, using Mean Value Theorem (MVT) on $\Theta,$ we derive that there exists $z_1 \in e_{a}$ such that $\Theta(\alpha_{z_1}) = 0.$ In other words, $P \in \Rh_{a, b},$ which completes the proof of the proposition for $a > b$ when \ref{item:propI} holds.  

For case \ref{item:propII}, note that \[0 \leq \im(P), \im(C_{z}) \leq a - a^{-1}, \quad \quad \mbox{for all} \ z \in e_{a},\] where $C_{z}$ $(= z)$ is the centre of $E_{b, z}$ given in \ref{item:propA}. Then, the \textit{continuous property} of $\im\left(C_{ z}\right)$ mentioned in \eqref{imaginaryzh} implies that there exists a $z'' \in e_{a}$ such that $\im(P) = \im(C_{z''})$ (see Figure \ref{region5}). Let $L_{\im(P)} : t = \im(P)$ denote the line joining $P$ and $C_{z''},$ and let $L_{\im(P)}$ intersect $E^{a, b}$ at $P_1$ in $Q_{+, +}$ such that any $t \in L_{\im(P)}$ satisfying $\re(t) > \re(P_1)$ lies in $U_{a, b}.$ Then $P_1$ has a representation as in \eqref{expressr}, namely \[
 P_1 = \left(a + a^{-1}\right)\cos \alpha + \left(b + b^{-1}\right)\cos \beta + i\left[\left(a - a^{-1}\right)\sin \alpha + \left(b - b^{-1}\right)\sin \beta\right],\] for some $\alpha, \beta \in [0, \pi/2).$ Moreover, there exist $\gamma \in [0, \pi/2)$ such that $\im(P) = \im (P') = \im(C_{z''}) = \left(a - a^{-1}\right)\sin \gamma.$ The case $\re(P) = \re(P_1)$ is trivial as the above discussion implies that $P = P_1.$ If $\re(P) > \re(P_1),$ then from the definition of $P_1$ it follows that $P \in U_{a, b}.$  Therefore, it only remains to investigate the case when $0 \leq \re(P) < \re(P_1).$

\begin{figure}[h]
\centering
{\includegraphics[width=80mm,scale=.4]{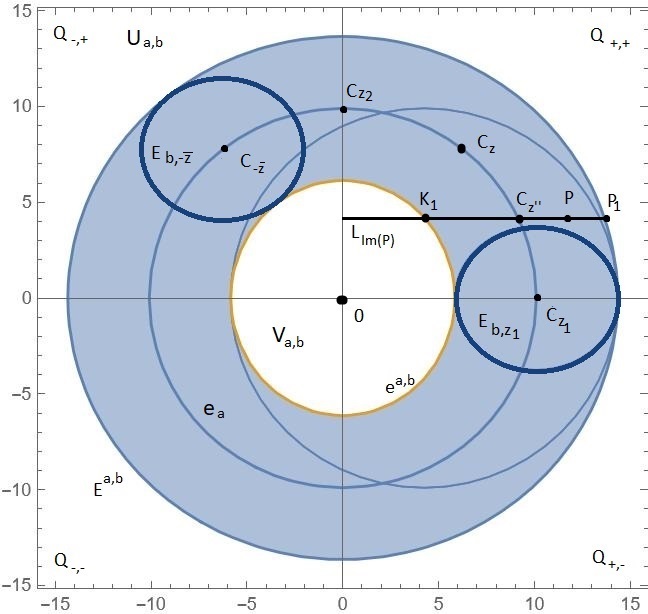}}   \ {\includegraphics[width=80mm,scale=.4]{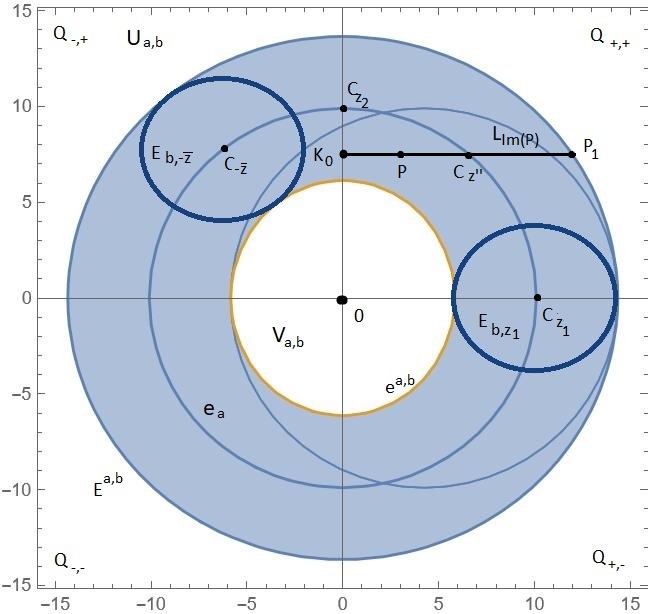}} \\ \ \ \quad (a) \quad \quad \quad \quad \quad \quad \quad \quad \quad \quad \quad \quad \quad \quad \quad \quad \quad \quad \   (b)
\caption{$\Rh_{a, b}, U_{a, b}, V_{a, b}, e_{a}, C_{z}, P, P_1, L_{\im(P)}$ when $a > b$ and \ref{item:propII} holds; left (resp. right) figure shows the case when $L_{\im(P)} \cap e^{a, b} \neq \varnothing$ (resp. $L_{\im(P)} \cap e^{a, b} = \varnothing$).}\label{region5}
\end{figure}

Define the function $f : \left[\alpha, \gamma\right] \rightarrow \left[0, \beta\right],$ which sends $\psi$ to $f(\psi)$ such that \[\left(a - a^{-1}\right)\sin \gamma = \left(a - a^{-1}\right)\sin \alpha + \left(b - b^{-1}\right)\sin \beta = \left(a - a^{-1}\right)\sin \psi + \left(b - b^{-1}\right)\sin f(\psi).\] Note that this is a continuous onto function from a connected set. Then the graph of this function, namely $\Gamma_f := \{(\psi, f(\psi)): \psi \in \left[\alpha, \gamma\right]\},$ is also a connected set. Consider another function $g: \Gamma_f \rightarrow \R,$ defined by \[(\psi, f(\psi)) \mapsto \left(a + a^{-1}\right)\cos \psi + \left(b + b^{-1}\right)\cos f(\psi).\] Note that this function is a continuous function, and there exist $\chi, \xi \in \Gamma_f$ such that $g(\chi) = \re(P_1) = \left(a + a^{-1}\right)\cos \alpha + \left(b + b^{-1}\right)\cos \beta,$ and $g(\xi) =  \re\left(C_{z''}\right) = \left(a + a^{-1}\right)\cos \gamma.$ 

Now if $\re(P) \in \left(\re\left(C_{z''}\right), \re(P_1)\right),$ we claim that there exists a $\psi_0 \in \left[\alpha, \gamma\right]$ such that \[P = \left(a + a^{-1}\right)\cos \psi_0 + \left(b + b^{-1}\right)\cos f(\psi_0) + i\left[\left(a - a^{-1}\right)\sin \psi_0 + \left(b - b^{-1}\right)\sin f(\psi_0)\right],\] which will imply that $P \in \Rh_{a, b}.$ Indeed, $g$ is a continuous function on a connected set, and $g(\xi) < g(\chi).$ Therefore, all the values of the interval $(g(\xi), g(\chi))$ are attained by $g.$ In particular, such $\psi_0$ exists. This proves the statement of the proposition when $\im(P) \in \left[0, a - a^{-1}\right]$ and $\re(P) \geq \re\left(C_{z''}\right).$

Therefore, it remains to consider the case when \ref{item:propII} holds along with $\re(P) \in \left[0,  \re\left(C_{z''}\right)\right).$ 

If the line $L_{\im(P)}$ intersect $e^{a, b}$ in $Q_{+, +},$ then we consider the intersection point with the smallest non-negative real part. In other words, if $K_1, \dots, K_l \in Q_{+,+} \cap e^{a, b} \cap  L_{\im(P)}$ are distinct with $0 \leq \re(K_1) < \dots < \re(K_l),$ then consider $\re(K_1).$ We want to show that, in fact there can be at most one intersection point of $L_{\im(P)}$ and $e^{a, b}$ in $Q_{+, +}.$ That is, if $\re(P) \in (\re(K_1), \re(C_{z''})),$ then $P \in \Rh_{a, b};$ but this follows from a similar argument as above. Therefore, it only remains to investigate the case when $L_{\im(P)}$ does not intersect $e^{a, b}$ in $Q_{+, +}.$ Let $K_0$ be the intersection point of $L_{\im(P)}$ and the imaginary axis. Then we have $\re(K_0) = 0,$ and $\re(P) \in (\re(K_0), \re(C_{z''})) \subset L_{\im(P)}.$ Again, an analogous argument as above implies that $P \in \Rh_{a, b}.$

We collect all the results above, and then using the symmetry of $\Rh_{a, b}$ (see \ref{item:propB}) we conclude that if $P \notin U_{a, b} \cup V_{a, b},$ then $P \in \Rh_{a, b},$ which completes the proof of the proposition for $a > b.$ The case $a = b$ follows from a similar argument.
\end{proof}

\subsubsection{\textbf{Application of Theorem \ref{constant} to the bounded component}}

Now we are ready to apply Theorem \ref{constant} to $V_{a, b},$ and evaluate $\m_{a, b}(Q_r).$ Firstly, we need to investigate the roots of $xQ_{r_0}(x, b).$ Since, $0 \in V_{a, b},$ we can choose $r_0 = 0$ in our theorem. In particular, we need to count the number of roots of $xQ_0(x, b)$ lying inside the circle $|x| = a.$ By Lemma \ref{Lem1}, we can also assume $a > b > 1.$

Factoring $xQ_0(x, b)$ in $\C[x],$ we obtain that \begin{equation*}
xQ_0(x, b) = x^2 + \left(b + \frac{1}{b}\right)x + 1 = \left(x + b\right)\left(x + \frac{1}{b}\right).
\end{equation*} Since $a > b > 1,$ both roots of $xQ_0(x, b))$ lies inside the circle $|x| = a.$ Also note that $Q_{F, 0}^x(y)$ and $Q_{f, 0}^x(y)$ in \eqref{expansn1} are equal to the constant function $\textbf{1}.$ Applying Theorem \ref{constant}, we have, for $a > b > 1$ and $r \in V_{a, b},$ \[\m_{a, b}(Q_r) = \nu_{a, b, 0}^1 \log a = \log a,\] where the last equality follows from the fact that \[\nu_{a, b, 0}^1 = Z_{a, b, 0}^1 - P_{a, b, 0}^1 = 2 - 1 = 1.\] Other cases, such as $b > a > 1,$ $a > 1 > b$ etc, follow from a combination Lemma \ref{Lem1} and a similar arguments as above.

\section{Generalized Mahler measure of $X + \frac{1}{X} + Y + \frac{1}{Y} + 4$}   

In this section, our goal is to provide a proof of Theorem \ref{k=4*}, and evaluate \[\m_{a, b}(Q_4):= \m_{a, b}(Q_4(x, y)) = \m_{a, b}\left(x + \frac{1}{x} + y + \frac{1}{y} + 4\right)\] for all $a, b > 0.$ 

Our method of proof is mostly inspired from the proof of Theorem $12$ in \cite{CKL19}. We apply the change of variables considered by Boyd (see Section $2$A in \cite{Bo2}), namely \[x \mapsto \frac{w}{z} \quad \mbox{and} \quad y \mapsto wz,\] to $Q_4 (x, y),$ and this yields that \begin{equation}\label{eq:CoV}
  P\left(w, z\right) = Q_4\left(\frac{w}{z}, wz \right) = \frac{1}{wz} \left( 1 + iw + iz + wz\right)\left( 1 - iw - iz + wz\right). 
\end{equation}

Since $\m_{a, b}(S(x, y)T(x, y)) = \m_{a, b}(S(x, y)) + \m_{a, b}(T(x, y)),$ it is sufficient to evaluate the Mahler measures of the linear polynomials $\left( 1 \pm iw \pm iz + wz\right)$ over $\TT_{c, d}^2 = \{ (w, z) \in \C^* \times \C^* : |w| = c, |z| = d\},$ where \[c = \sqrt{ab}, \ d = \sqrt{\frac{b}{a}}.\] Afterwards, using the changes of variables, we can evaluate $\m_{a, b}(Q_4).$ The changes of variables \[w \mapsto -w \quad \mbox{and} \quad z \mapsto -z\] transform $\left( 1 + iw + iz + wz\right)$ to $\left( 1 - iw - iz + wz\right).$ As these changes of variables preserve the Mahler measure, we find that \begin{align}
\m_{a, b}(Q_4) = \m_{c, d}(P(w, z)) =& \m_{c, d}\left(\frac{1}{wz}\right) + \m_{c, d}\left( 1 + iw + iz + wz\right) + \m_{c, d}\left( 1 - iw - iz + wz\right) \nnum \\ =& -\log cd + 2\m_{c, d}\left( 1 + iw + iz + wz\right) \nnum \\ =& -\log b + 2\m_{c, d}\left( 1 + iw + iz + wz\right), \label{eq:factor4}
\end{align} 
where the last equality follows from the fact that $cd = \sqrt{ab} \cdot \sqrt{\frac{b}{a}} = b.$ 

Among the terms in \eqref{eq:factor4}, it remains to evaluate \[\frac{1}{2}\left(\m_{c, d}(P) + \log cd\right) = \frac{1}{2}\left(\m_{a, b}(Q_4) + \log b\right) = \m_{c, d}(1+iw+iz + wz).\] 

Note that $z(w) = -\frac{1+iw}{i+w}$ is the only root of $R(w, z) = 1+iw+iz + wz,$ when considered as a polynomial in $z.$ Therefore, \begin{align}
\m_{c, d}(R(w, z)) &= \m_{c, d}(w + i) +  \m_{c, d}\left(z + \frac{1+iw}{i+w} \right) \nnum \\ &= \frac{1}{(2\pi i)^2}\int_{\TT_{c, d}^2}\log |w + i| \frac{dw}{w}\frac{dz}{z} +  \m_{c, d}\left(z + \frac{1+iw}{i+w} \right) \nnum \\ &= \frac{1}{2\pi i}\int_{|w| = c}\log |w + i| \frac{dw}{w} + \m_{c, d}\left(z + \frac{1+iw}{i+w} \right). \label{evalmaincd}
\end{align}
To evaluate the first integral, we apply the change of variables $w = cw'$ and Jensen's formula (see \eqref{jensen}) to obtain \begin{align}
    \frac{1}{2\pi i}\int_{|w| = c} \log|w+i| \frac{dw}{w} = \log c + \frac{1}{2\pi i}\int_{|w'| = 1} \log\left|w' + \frac{i}{c} \right| \frac{dw'}{w'} =& \left\{\begin{array}{ll} \log c \quad &\mbox{if} \ c > 1, \\ 0 \quad &\mbox{if} \ c \leq 1. \end{array}\right.  \label{intlog}
\end{align}

 It now suffices to evaluate \begin{align}
\m_{c, d}\left(z + \frac{1+iw}{i+w} \right) &= \frac{1}{(2\pi i)^2}\int_{\TT_{c, d}^2}\log \left|z +\frac{1+iw}{i+w} \right| \frac{dw}{w}\frac{dz}{z} \label{cdpart2} \\ &= \frac{1}{2\pi i}\int_{|w| = c}\left(\frac{1}{2\pi i}\int_{|z| = d}\log \left|z +\frac{1+iw}{i+w} \right| \frac{dz}{z}\right)\frac{dw}{w} \nnum
\end{align} to complete the proof. Note that $\frac{1}{2\pi i}\int_{|z| = d}\log \left|z + \frac{1+iw}{i+w} \right| \frac{dz}{z}$ can be simplified to \begin{equation}\label{simplify}
\frac{1}{2\pi i}\int_{|z| = d}\log \left|z +\frac{1+iw}{i+w} \right| \frac{dz}{z} =  \left\{\begin{array}{ll} \log  \left|\frac{1+iw}{i+w} \right| \quad &\mbox{if} \ \left|\frac{1+iw}{i+w} \right| > d, \\ \\ \log d \quad &\mbox{if} \ \left|\frac{1+iw}{i+w}\right| \leq d \end{array}\right.
\end{equation} following an application of Jensen's formula. 

Let $\gamma_{> d}$ and $\gamma_{\leq d}$ be the two collections of arcs defined by \[\gamma_{> d} = \{w: |w| = c, |z(w)| > d\}, \quad \quad \gamma_{\leq d} = \{w: |w| = c, |z(w)| \leq d\}.\]Then, applying Jensen's formula with respect to the variable $z,$ \eqref{cdpart2} can be expressed as \begin{align}
\m_{c, d}\left(z + \frac{1+iw}{i+w} \right) &= \frac{1}{2\pi i}\int_{|w| = c}\left(\frac{1}{2\pi i}\int_{|z| = d}\log \left|z + \frac{1+iw}{i+w}\right| \frac{dz}{z}\right)\frac{dw}{w} \nnum \\ &= \frac{1}{2\pi i}\int_{\gamma_{> d}}\log  \left|\frac{1+iw}{i+w} \right| \frac{dw}{w} + \frac{1}{2\pi i}\int_{\gamma_{\leq d}}\log d \frac{dw}{w}. \label{cdpart2.1}
\end{align}
Since $\im \left(\frac{dw}{w}\right) = d \arg w,$ the differential form can be represented in terms of $\eta$ as \[\log  \left|\frac{1+iw}{i+w} \right| \frac{dw}{w} = \log|z(w)|\frac{dw}{w} =-i\left(\eta(w, z(w)) - \eta(c, z(w))\right).\] The second term above can be further simplified to \[\eta(c, z(w)) = \eta(c, iz(w)) - \eta(c, i) = \eta(c, iz(w)) = (\log c)  d\arg \left(  \frac{1 + iw}{1 - iw} \right), \] where $iz(w) = i\frac{1+iw}{i+w} = \frac{1 + iw}{1-iw}.$ 
Therefore, once we have determined $\gamma_{> d}$ and $\gamma_{\geq d}$ explicitly, the integrals in \eqref{cdpart2.1} can be evaluated individually using the properties of $\eta$ and the following two lemmas.
\begin{lem}\label{lemma1.3}
For $w, z(w)$ mentioned above, $\eta\left(w, z(w)\right)$ decomposes as \[\eta(w, z(w)) = \eta\left(-iw, 1 + iw\right) - \eta\left(iw, 1 - iw\right).\] 
\end{lem}

\begin{lem}[Lemma 16, \cite{CKL19}]\label{lemma1.4}
For $c \in \R_{> 0}$ and $\theta \in [-\pi, \pi),$ let $w = ce^{i\theta}$ and $\psi = \theta + \frac{\pi}{2}.$ Then \[ d \arg \left(  \frac{1 + iw}{1 - iw} \right) = \frac{2\left(c^{-1} - c\right)\cos\psi }{\left(c^{-1} - c\right)^2 + 4\sin^2\psi} d \psi.\]
\end{lem}

Using property \eqref{eta(x,1-x)} of $\eta$ we can rewrite $\eta(w, z(w))$ in Lemma \ref{lemma1.3} as \begin{equation}\label{etaD}
\eta(w, z(w)) = d D\left(-iw\right) - d D \left(iw\right),
\end{equation} where $D$ is the Bloch--Wigner dilogarithm given in \eqref{Bloch-Wigner}.

The evaluation of the remaining integral involving $\eta(c, z(w))$ ($= \log c \ d \arg \left(  \frac{1 + iw}{1 - iw} \right)$) over the integration path $\gamma_{> d}$  follows from the lemma below.   

\begin{lem}\label{lemma1.5}
For $c \in \R_{> 0}$ and $\theta \in [-\pi, \pi),$ let $w = ce^{i\theta}.$ Let $\alpha, \beta \in [-\pi, \pi).$ Then \[\int_{w(\alpha)}^{w(\beta)}  d \arg \left(  \frac{1 + iw}{1 - iw} \right)  =  \tan^{-1}\left(\frac{2\cos\alpha}{c - c^{-1}}\right) - \tan^{-1} \left(\frac{2\cos\beta}{c - c^{-1}}\right),\] where $w(\alpha) = ce^{i\alpha}$ and $w(\beta) = ce^{i\beta}.$
\end{lem}

We omit the proof of Lemma \ref{lemma1.4} since it is an intermediate step in Lemma $16$ of \cite{CKL19}. We should also remark that Lemma \ref{lemma1.5} is a generalized version of Lemma $16$ in \cite{CKL19}, which states the above result for the case $\alpha = -\pi$ and $\beta = 0.$ We will see later that the proof of Lemma \ref{lemma1.5} also follows from an argument similar to the proof in \cite{CKL19}. We now provide the proofs of Lemma \ref{lemma1.3} and \ref{lemma1.5}.

\begin{proof}[\textbf{Proof of Lemma \ref{lemma1.3}}]
Using properties of $\eta$ in Lemma \ref{lem:eta}, $\eta (w, z(w))$ decomposes as \begin{align*}
\eta(w, z(w)) &= \eta \left(w, \frac{1+iw}{i+w}\right) \\ &= \eta (w, 1 + iw) - \eta(w, i + w) \\ &= \eta(-iw, 1 + iw) - \eta(-i, 1 + iw) - \eta(iw, i + w) + \eta(i, i + w) \\ &= \eta(-iw, 1+iw) - \eta(iw, 1 - iw) - \eta(iw, i) \\ &= \eta(-iw, 1 + iw) - \eta(iw, 1-iw),
\end{align*} where we applied Remark \ref{remark2}, which implies that $\eta(\zeta, f(w)) = 0 = \eta(f(w), \zeta)$ for any root of unity $\zeta$ and any function $f(w)$ of $w.$  
\end{proof}

 \begin{proof}[\textbf{Proof of Lemma \ref{lemma1.5}}:] We first assume that $c \in \R_{> 0} \setminus \{1\},$ and the case $c = 1$ follows from a continuity argument. 
 
For $\alpha, \beta \in [-\pi, \pi)$ and $w = ce^{i\theta},$ Lemma \ref{lemma1.4} yields that \[\int_{w(\alpha)}^{w(\beta)}  d \arg \left(  \frac{1 + iw}{1 - iw} \right)  = \int_{\alpha + \frac{\pi}{2}}^{\beta + \frac{\pi}{2}} \frac{2\left(c^{-1} - c\right)\cos\psi }{\left(c^{-1} - c\right)^2 + 4\sin^2\psi} d \psi = - \int_{\cos \alpha}^{\cos \beta} \frac{2(c - c^{-1})}{(c - c^{-1})^2 + 4t^2}dt, \] where $\psi = \theta + \frac{\pi}{2},$ $w(\phi) = ce^{i\phi},$ and the last equality follows from the change of variables $\sin \psi \mapsto t.$

Further, the change of variables $\frac{2t}{c - c^{-1}} \mapsto u$ gives that \begin{equation*}
\int_{w(\alpha)}^{w(\beta)}  d \arg \left(  \frac{1 + iw}{1 - iw} \right) = - \int_{\frac{2 \cos \alpha}{c - c^{-1}}}^{\frac{2 \cos \beta}{c - c^{-1}}} \frac{d u}{1 + u^2} = \tan^{-1}\left(\frac{2\cos\alpha}{c - c^{-1}}\right) - \tan^{-1} \left(\frac{2\cos\beta}{c - c^{-1}}\right),
\end{equation*} which proves the lemma.
\end{proof}

Now we have everything to complete the proof of Theorem \ref{k=4*}.

\begin{proof}[\textbf{Proof of Theorem \ref{k=4*}}] In order to apply Lemma \ref{lemma1.3} and Lemma \ref{lemma1.4} to \eqref{cdpart2.1}, it is necessary to explicitly express $\gamma_{\leq d}$ and $\gamma_{> d}.$  Since $\gamma_{> d}$ and $\gamma_{\leq d}$ are disjoint, and \[\{w : |w| = c\} = \gamma_{> d} \cup \gamma_{\leq d},\] it suffices to understand $\gamma_{> d}.$

Recall that $z(w) = -\frac{1+iw}{i+w}.$ Then, $|z(w)| > d \Leftrightarrow \left|\frac{1+iw}{i+w}\right| > d \Leftrightarrow \left|1 + iw\right| > d\left|i+w\right|.$ Since both sides of the inequality are non-negative, we can square them and get \begin{align*}
\left|1 + iw\right| > d\left|i+w\right| &\Leftrightarrow \left|1 + iw\right|^2 > d^2\left|i+w\right|^2 \\  &\Leftrightarrow 2 (1 + d^2)\Re(iw) > (d^2 - 1)(1 + |w|^2) \\ &\Leftrightarrow \Re(ie^{i\theta}) > \frac{d^2 - 1}{1 + d^2}\cdot \frac{1+c^2}{2c},
\end{align*}
where the last inequality follows from the fact that $w = ce^{i\theta}$ for $\theta \in [-\pi, \pi).$ In other words, the condition $|z(w)| > d$ is equivalent to the condition \begin{equation}\label{condition}
-1 \leq - \re(ie^{i\theta}) = \sin \theta < \frac{1 - d^2}{1 + d^2}\cdot \frac{1+c^2}{2c}, 
\end{equation} with $\theta \in [-\pi, \pi).$  For simplicity we denote \[\mathcal{A}_{c, d} :=  \frac{1 - d^2}{1 + d^2}\cdot \frac{1+c^2}{2c}.\] As $|\sin \theta| \leq 1,$ there are three cases to consider. \begin{itemize}
\item[\hypertarget{label}{\textbf{Case 1}}:] If $\mathcal{A}_{c, d} \leq -1,$ then $\gamma_{> d} = \varnothing \quad \mbox{and} \quad \gamma_{\leq d} = \{w: |w| = c\}.$

\medskip

\item[\hypertarget{label}{\textbf{Case 2}}:] If $\mathcal{A}_{c, d} \geq 1,$ then $\gamma_{> d} = \{w: |w| = c\}  \quad \mbox{and} \quad \gamma_{\leq d} = \varnothing.$

\medskip

\item[\hypertarget{label}{\textbf{Case 3}}:] If $\left|\mathcal{A}_{c, d}\right| < 1,$ then, for $w = ce^{i\theta}$ with $\theta \in [-\pi, \pi),$ \newline $\gamma_{> d} = \left\{w: |w|= c, -1 \leq \sin \theta < \mathcal{A}_{c, d}\right\}$ and $\gamma_{\leq d} = \left\{w: |w|= c, \mathcal{A}_{c, d} \leq  \sin \theta \leq 1\right\}.$
\end{itemize}

Now we have everything needed to evaluate \eqref{cdpart2.1}, namely \[\m_{c, d}\left(z + \frac{1+iw}{i+w} \right) = \frac{1}{2\pi i}\int_{\gamma_{> d}}\log  \left|\frac{1+iw}{i+w} \right| \frac{dw}{w} + \frac{1}{2\pi i}\int_{\gamma_{\leq d}}\log d \frac{dw}{w}.\]

\hyperlink{label}{\textbf{Case 1}}: Since $\gamma_{> d} = \varnothing,$ the integrals in \eqref{cdpart2.1} can be evaluated individually to obtain \[\frac{1}{2\pi i}\int_{\gamma_{> d}}\log  \left|\frac{1+iw}{i+w} \right| \frac{dw}{w} = 0, \quad \mbox{and} \quad \frac{1}{2\pi i}\int_{\gamma_{\leq d}}\log d \frac{dw}{w} = \frac{1}{2\pi i}\int_{|w| = c}\log d \frac{dw}{w} = \log d.\] Therefore, in this case we have \begin{equation}\label{4case1}
\m_{c, d}\left(z + \frac{1+iw}{i+w} \right) = \log d.
\end{equation} 

\medskip

\hyperlink{label}{\textbf{Case 2}}: Since $\gamma_{\leq d} = \varnothing,$ the second integral in \eqref{cdpart2.1} contributes nothing. However the first integral can be decomposed into simpler integrals, i.e. \begin{align}
\frac{1}{2\pi i}\int_{\gamma_{> d}}\log  \left|\frac{1+iw}{i+w} \right| \frac{dw}{w} &= \frac{1}{2\pi i}\int_{|w| = c}\log  \left|\frac{1+iw}{i+w} \right| \frac{dw}{w} \nnum \\ &= \frac{1}{2\pi i}\int_{|w| = c}\log  \left|1 +  iw \right| \frac{dw}{w} - \frac{1}{2\pi i}\int_{|w| = c}\log  \left|i + w \right| \frac{dw}{w} \nnum  &= 0. \nnum
\end{align} 
Therefore, when $\gamma_{> d} = \{|w| = c\}$ (and $\gamma_{ \leq d} = \varnothing$), then \begin{equation}\label{4case2}
\m_{c, d}\left(z + \frac{1+iw}{i+w} \right) = 0.
\end{equation}

\medskip

\hyperlink{label}{\textbf{Case 3}}: Since \[\left|\mathcal{A}_{c, d}\right| < 1,\]  we have two sub-cases to consider. 

\hypertarget{label}{\textbf{\textsc{3a}}}: When \[ -1 < \mathcal{A}_{c, d} < 0,\] then $\sin^{-1}\left(\mathcal{A}_{c, d}\right) \in [-\pi, 0).$ For simplicity, we denote $ \tau = \sin^{-1}\left(\mathcal{A}_{c, d}\right)$ such that $\tau \in \left(-\frac{\pi}{2}, 0\right).$ Note that $\sin \tau = \sin (-\pi - \tau).$ Then the boundary values of $\gamma_{> d}$ are \[\partial \gamma_{> d} = \{w(-\pi - \tau), w(\tau)\} = \{ce^{i(-\pi - \tau)}, ce^{i\tau}\} = \{-ce^{-i\tau}, ce^{i\tau}\},\] where $w(\theta) = ce^{i\theta}.$ The integration path $\gamma_{\leq d}$ is then the union of the arcs joining $w(-\pi)$ and  $w(-\pi - \tau),$ and joining $w(\tau)$ and $w(\pi).$ Therefore \[\partial \gamma_{\leq d} = \{w(-\pi), w(-\pi - \tau), w(\tau), w(\pi)\}\] are the boundary values of $\gamma_{\leq d}.$ All the paths are assumed to be traversed counter-clockwise. Now, we have all the tools to calculate \eqref{cdpart2.1} in this case. Combining Lemma \ref{lemma1.3} and \eqref{etaD} with the above discussion we obtain \begin{align}
\m_{c, d}\left(z + \frac{1+iw}{i+w} \right) =& \frac{1}{2\pi i}\int_{\gamma_{> d}}\log  \left|\frac{1+iw}{i+w} \right| \frac{dw}{w} + \frac{1}{2\pi i}\int_{\gamma_{\leq d}}\log d \frac{dw}{w} \nnum \\ =& -\frac{1}{2\pi}\int_{\gamma_{> d}} \eta(w, z(w)) + \frac{1}{2\pi} \int_{\gamma_{> d}} \eta(c, z(w)) + \frac{1}{2\pi i}\int_{\gamma_{\leq d}}\log d \frac{dw}{w} \nnum \\ =&  -\frac{1}{2\pi}\int_{\gamma_{> d}} (dD(-iw) - dD(iw))  \nnum \\  &+ \frac{\log c}{2\pi}\int_{w(-\pi - \tau)}^{w(\tau)}  d \arg \left(  \frac{1 + iw}{1 - iw} \right) + \frac{\log d}{2\pi}\left(\int_{-\pi}^{-\pi -\tau} + \int_{\tau}^{\pi}\right) d\theta, \label{evalcd3a} 
\end{align} where the simplification of the last integral follows from the above discussion regarding $\partial \gamma_{\leq d}$ and substituting $w$ with $w(\theta) = ce^{i\theta}.$ The first integral in \eqref{evalcd3a} can be evaluated using Stokes' theorem as \begin{align}
\frac{1}{2\pi}\int_{\gamma_{> d}} (dD(-iw) - dD(iw)) =&  \frac{1}{2\pi}\left[D(-iw) - D(iw)\right]_{\partial \gamma_{> d}} \nnum \\ =&  \frac{1}{2\pi}\left[D(-iw) - D(iw)\right]_{w(-\pi - \tau)}^{w(\tau)} \nnum \\ =& -\frac{1}{\pi}\left( D(ice^{-i\tau}) + D(ice^{i\tau})\right), \label{evalcd3a1}
\end{align} where the last equality follows from the property \eqref{eq:relations} of the Bloch-Wigner dilogarithm. 
 
Substituting $\alpha = -\pi - \tau$ and $\beta = \tau$ in the statement of Lemma \ref{lemma1.5}, we evaluate the second integral in \eqref{evalcd3a}: \begin{equation}\label{evalcd3a2}
\frac{\log c}{2\pi}\int_{w(-\pi - \tau)}^{w(\tau)}  d \arg \left(  \frac{1 + iw}{1 - iw} \right) = - \frac{\log c}{\pi} \tan^{-1} \left(\frac{2\cos\tau}{c - c^{-1}}\right).
\end{equation} 
The remaining integral's contribution is \begin{equation}\label{evalcd3a3}
\frac{\log d}{2\pi}\left(\int_{-\pi}^{-\pi -\tau} + \int_{\tau}^{\pi}\right) d\theta = \frac{\log d}{2\pi} \left[-\pi - \tau + \pi + \pi - \tau\right] = \frac{\pi - 2\tau}{2\pi} \log d.
\end{equation} 
Then \eqref{evalcd3a1}, \eqref{evalcd3a2} and \eqref{evalcd3a3} together yield that \begin{equation*}
\m_{c, d}\left(z + \frac{1+iw}{i+w} \right) = \frac{1}{\pi}\left[ D(ice^{-i\tau}) + D(ice^{i\tau})- (\log c) \tan^{-1} \left(\frac{2\cos\tau}{c - c^{-1}}\right) + \left(\frac{\pi}{2} - \tau\right)\log d\right].
\end{equation*}

\medskip

\hypertarget{label}{\textbf{\textsc{3b}}}: It remains to evaluate the case when \[0 < \mathcal{A}_{c, d} < 1.\] This condition is equivalent to \[\sin^{-1}\left(\mathcal{A}_{c, d}\right) \in (0, \pi).\] Again, for simplicity, we denote $\kappa = \sin^{-1}\left(\mathcal{A}_{c, d}\right)$ such that $\kappa \in \left(0, \frac{\pi}{2}\right).$ Since $\sin \kappa = \sin (\pi - \kappa)$ and $\sin \pi = 0,$ the boundary values in this case are \[\partial \gamma_{> d} = \{w(-\pi), w(\kappa), w(\pi - \kappa), w(\pi)\},\] and \[\partial \gamma_{\leq d} = \{w(\kappa), w(\pi - \kappa)\}.\] The arcs are considered to be oriented in a counter-clockwise direction. From a similar argument as before, we deduce that \begin{equation*}
\m_{c, d}\left(z + \frac{1+iw}{i+w} \right) = \frac{1}{\pi}\left[ D(ice^{-i\kappa}) + D(ice^{i\kappa})- (\log c) \tan^{-1} \left(\frac{2\cos\kappa}{c - c^{-1}}\right) + \left(\frac{\pi}{2} - \kappa\right)\log d\right].
\end{equation*}

 We combine the results obtained in \textbf{\textsc{3a}} and \textbf{\textsc{3b}} to obtain  \begin{equation}\label{4case3}
 \m_{c, d}\left(z + \frac{1+iw}{i+w} \right) = \frac{1}{\pi}\left[ D(ice^{-i\mu}) + D(ice^{i\mu})- (\log c) \tan^{-1} \left(\frac{2\cos\mu}{c - c^{-1}}\right) + \left(\frac{\pi}{2} - \mu\right)\log d \right],
 \end{equation}
 where $\mu =  \sin^{-1}\left(\mathcal{A}_{c, d}\right) \in \left(-\frac{\pi}{2}, \frac{\pi}{2}\right).$ This concludes the evaluation of $\m_{c, d}\left(z + \frac{1+iw}{i+w} \right)$ for the different cases.

Recall that $R(w, z) = 1 + iw + iz + wz.$ In order to evaluate $\m_{c, d}(R(w, z)),$ it suffices to collect the equalities in \eqref{intlog}, \eqref{4case1}, \eqref{4case2} and \eqref{4case3}. We deduce \[\m_{c, d}(R(w, z)) = \max\{\log c, 0\} + \left\{\begin{array}{ll} \max\{\log d, 0\}  \quad &\mbox{if} \ \left|\mathcal{A}_{c, d}\right| \geq 1, \\ \\ \frac{1}{\pi}\left[ D(ice^{-i\mu}) + D(ice^{i\mu}) \right. \\  \left. - (\log c) \tan^{-1} \left(\frac{2\cos\mu}{c - c^{-1}}\right)+ \left(\frac{\pi}{2} - \mu\right)\log d \right] \quad &\mbox{if} \ \left|\mathcal{A}_{c, d}\right| < 1, \end{array}\right.,\] where $\mu =  \sin^{-1}\left(\mathcal{A}_{c, d}\right) \in \left(-\frac{\pi}{2}, \frac{\pi}{2}\right).$ Notice that $\m_{c, d}(R(w, z)) = \frac{1}{2}\left(\m_{a, b}(Q_4) + \log b\right),$ where $a = \frac{c}{d}$ and $b = cd.$ This implies that when $\left|\mathcal{A}_{c, d}\right| \geq 1,$ we have \begin{equation}\label{grt1}
\m_{a, b}(Q_4) = \max\{\log c, -\log c\} + \max\{\log d, -\log d\}.
\end{equation} On the other hand, when $\left|\mathcal{A}_{c, d}\right| < 1,$ \begin{align*}
\m_{a, b}(Q_4) =& \max\{\log c, -\log c\} - \log d + \frac{2}{\pi}\left[ D(ice^{-i\mu}) + D(ice^{i\mu}) \right] \\ &- \frac{2\log c}{\pi}\tan^{-1} \left(\frac{2\cos\mu}{c - c^{-1}}\right) + \left(1 - \frac{2\mu}{\pi}\right)\log d  \\ =& \frac{2}{\pi}\left[ D(ice^{-i\mu}) + D(ice^{i\mu}) - \mu \log d\right] + \frac{2\log c}{\pi}\left[\max\left\{\frac{\pi}{2}, -\frac{\pi}{2}\right\} - \tan^{-1} \left(\frac{2\cos\mu}{c - c^{-1}}\right)\right] \\ =& \frac{2}{\pi}\left[ D(ice^{-i\mu}) + D(ice^{i\mu}) - \mu \log d\right] + \frac{2\log c}{\pi}\tan^{-1}\left(\frac{c - c^{-1}}{2\cos \mu}\right),
\end{align*} where $\mu = \sin^{-1}\left(\mathcal{A}_{c, d} \right) \in \left(-\frac{\pi}{2}, \frac{\pi}{2}\right),$ and the simplification of the last term follows from the fact that  
\[\mbox{if} \ x > 0, \quad \pi/2 - \tan^{-1}(x) = \tan^{-1}\left(x^{-1}\right),\] and \[\mbox{if} \ x < 0, \quad -\pi/2 - \tan^{-1}(x) = \tan^{-1}\left(x^{-1}\right).\] Therefore, \eqref{grt1} along with the above discussion implies that \begin{equation*}
\m_{a, b}(Q_4) = \left\{\begin{array}{ll} \max\{\log c, -\log c\} + \max\{\log d, -\log d\} \quad &\mbox{if} \ \left|\mathcal{A}_{c, d}\right| \geq 1, \\ \\ \frac{2}{\pi}\left[ D(ice^{-i\mu}) + D(ice^{i\mu}) - \mu \log d + (\log c) \tan^{-1}\left(\frac{c - c^{-1}}{2\cos \mu}\right)\right] \quad &\mbox{if} \ \left|\mathcal{A}_{c, d}\right| < 1, \end{array}\right.
\end{equation*} where $\mu = \sin^{-1}\left(\mathcal{A}_{c, d} \right) \in \left(-\frac{\pi}{2}, \frac{\pi}{2}\right),$ which completes the proof of Theorem \ref{k=4*}.
\end{proof}

\section{Conclusion}

 There are several directions for further exploration. The most immediate question one can ask is how to evaluate $\m_{a, b}(Q_r)$ when $r \in \Rh_{a, b}.$ A primary observation in this case is that the integration path is not necessarily closed (and in most cases it is not). This turns out to be a challenging problem since the integration path in this case cannot be easily identified as a cycle in the homology group. We face a similar obstacle while evaluating the Mahler measure of $Q_r(x, y)$ on the bounded connected components when the number of roots of $Q_r(a, y)$ (counting multiplicity) (or $Q_{r}(x, b)$) inside $\TT_{b}^1$ (or $\TT_{a}^1$) is strictly less than the degree of the polynomials. In this situation it is frequently required to integrate the algebraic functions coming from the factorisation of $Q_r(x, y)$ (when considered as a polynomial in either $x$ or $y$), on paths which are not closed. These similar challenges also extend to the $n$-variable cases when $n \geq 3.$
 
 A different direction would be to consider the family of rational polynomials \[P_{k}(x_1, \dots, x_n) = k - \frac{P(x_1, \dots, x_n)}{Q(x_1, \dots, x_n)} \in \C(x_1, \dots, x_n), \quad \quad \mbox{for} \ k \in \C.\] Our method of proof for Theorems \ref{mainthm} and \ref{constant} extends to this type of rational polynomials when $Q(x_1, \dots, x_n)$ is a monomial, which essentially recovers Theorems \ref{nvarthm} and \ref{nvarconstant}. 
 
 The expression of $\nu_{\mathfrak{a}, k}^j$ in \eqref{nvarnu} appears in the work of Forsberg, Passare, and Tsikh \cite{FPT}, where it is denoted as the \textit{order} of an element in the complement of the \textit{Amoeba} associated to the respected polynomial. Our theorems also re-establish certain properties of the \textit{Ronkin function} associated to amoebas mentioned in \cite{FPT}. Therefore, it would be also natural to explore the generalized Mahler measure in terms of the Ronkin function associated to amoebas in more depth.

\bibliographystyle{abbrv}

\bibliography{References}

\end{document}